\documentclass[a4paper,twoside,12pt]{amsart}
\usepackage{mathtools}
\usepackage{indentfirst,amssymb,fullpage}
\usepackage[all]{xy}
\usepackage{amsthm}
\usepackage{aliascnt}
\usepackage{hyperref}
\usepackage{todonotes}
\newtheorem{theorem}{Theorem}[section]

\newtheorem*{acknowledgement*}{\protect\acknowledgementname}

\newaliascnt{setup}{theorem}
\newtheorem{setup}[setup]{Setup}
\aliascntresetthe{setup}

\newaliascnt{question}{theorem}

\aliascntresetthe{question}

\newaliascnt{lemma}{theorem}
\newtheorem{lemma}[lemma]{Lemma}
\aliascntresetthe{lemma}

\newaliascnt{conjecture}{theorem}

\aliascntresetthe{conjecture}

\newaliascnt{proposition}{theorem}
\newtheorem{proposition}[proposition]{Proposition}
\aliascntresetthe{proposition}

\newaliascnt{corollary}{theorem}

\aliascntresetthe{corollary}

\newaliascnt{problem}{theorem}

\aliascntresetthe{problem}

\newaliascnt{claim}{theorem}

\aliascntresetthe{claim}

\theoremstyle{definition}

\newaliascnt{definition}{theorem}
\newtheorem{definition}[definition]{Definition}
\aliascntresetthe{definition}

\newaliascnt{example}{theorem}

\aliascntresetthe{example}

\theoremstyle{remark}

\newaliascnt{remark}{theorem}
\newtheorem{remark}[remark]{Remark}
\aliascntresetthe{remark}

\newaliascnt{remarks}{theorem}

\aliascntresetthe{remarks}

\def\et{\textrm{\'et}}
\def\cris{\mathrm{cris}}
\def\Fil{{\mathrm{Fil}}}

\newcommand \Ainf[1]{\textbf{A}_{\textrm{inf}}(#1)}
\newcommand \MFbig[1]{\mathcal{MF}^{\mathrm{tor}}_{\mathrm{big}}(#1)}
\newcommand \MF[1]{\mathcal{MF}^{\mathrm{tor}}(#1)}
\newcommand \Pcris[1]{\mathbb{P}^{\text{cris}}(#1)}

\numberwithin{equation}{section}

\makeatletter
\@mparswitchfalse%
\makeatother
\normalmarginpar
\begin{document}
\title{Logarithmic Crystalline Representations}

\author{Zhenmou Liu}
\email{zhenmouliu@mail.ustc.edu.cn}
\address{School of Mathematical Sciences, University of Science and Technology of China, Hefei, Anhui 230026, PR China}

\author{Jinbang Yang}
\email{yjb@mail.ustc.edu.cn}
\address{School of Mathematical Sciences, University of Science and Technology of China, Hefei, Anhui 230026, PR China}

\author{Kang Zuo}
\email{zuok@uni-mainz.de}
\address{School of Mathematics and Statistics, Wuhan University, Luojiashan, Wuchang, Wuhan, Hubei, 430072, P.R. China.}
\address{Institut f\"ur Mathematik, Universit\"at Mainz, Mainz 55099, Germany}

\begin{abstract}
In 1989, Faltings proved the comparison theorem between \'etale cohomology and crystalline cohomology by studying Fontaine--Faltings modules and crystalline representations. In his paper, he mentioned that these modules and representations could be extended to the logarithmic context, but without detail. This note aims to explicitly present the construction of logarithmic Fontaine--Faltings modules and logarithmic crystalline representations.
\end{abstract}

\maketitle
\tableofcontents

\section*{Introduction}

Over the complex numbers, polarized variations of Hodge structures provide a natural bridge between Hodge theory and local systems.
They consist, roughly speaking, of a flat bundle together with a Hodge filtration satisfying Griffiths transversality and a compatible polarization.
Fontaine--Faltings modules may be viewed as a $p$-adic analogue of such structures, with the Frobenius structure playing an essential additional role.

Let $k$ be a perfect field of odd characteristic $p$, let $W\coloneqq W(k)$ be the ring of Witt vectors of $k$, and let $Y$ be a smooth separated $W$-scheme whose generic and special fibers are both nonempty and geometrically connected.
Denote by $\mathcal Y$ the $p$-adic formal completion of $Y$.
Fontaine--Laffaille \cite{FoLa82} and Faltings \cite{Fal89} introduced Fontaine--Faltings modules over $\mathcal Y$.
Roughly speaking, a Fontaine--Faltings module consists of a module equipped with an integrable connection, a decreasing filtration satisfying Griffiths transversality, and a Frobenius structure compatible with these data.
Faltings constructed a contravariant fully faithful functor
\[
\mathbb D:
\mathcal{MF}^{\nabla}_{[0,p-2]}(\mathcal Y/W)
\longrightarrow
\operatorname{Rep}_{\mathbb Z_p}
\bigl(\pi_1^{\et}(\mathcal Y_K)\bigr),
\]
which associates to a Fontaine--Faltings module a continuous representation of the \'etale fundamental group of the rigid generic fiber.
The representations arising in this way are called crystalline representations.

In \cite[p.~43]{Fal89}, Faltings stated that this theory extends to the logarithmic setting, and subsequently used the logarithmic theory in his comparison between \'etale and crystalline cohomology.
However, the details of this logarithmic construction were not written out.
%Recently, Abhinandan and Benois \cite{AB26} developed a rational theory of logarithmic crystalline representations for small affine algebras equipped with a logarithmic structure, using $(\varphi,\Gamma)$-modules and Wach modules.

In the present paper, we instead work from the integral viewpoint of logarithmic Fontaine--Faltings modules and construct the corresponding logarithmic $\mathbb D^{\log}$-functor.

Let $Z\subseteq Y$ be a relative simple normal crossing divisor, let $\mathcal Z$ be its $p$-adic formal completion with rigid generic fiber $\mathcal Z_K$, and put
\[
U=Y-Z,\qquad
\mathcal U=\mathcal Y-\mathcal Z,\qquad
\mathcal Y_K^\circ=\mathcal Y_K-\mathcal Z_K.
\]
We introduce logarithmic Fontaine--Faltings modules over $(\mathcal Y,\mathcal Z)$ by allowing logarithmic poles along $\mathcal Z$ in the connection.
Our main goal is to associate to such a module a continuous representation of $\pi_1^{\et}(\mathcal Y_K^\circ)$ whose restriction to $\mathcal U_K$ agrees canonically with the representation obtained from Faltings' original construction.

The construction is carried out locally on logarithmically small affine opens.
For such an affine open, we construct a logarithmic period ring $B^+_{\log}(\widehat R)$ and, after choosing a suitable Frobenius lifting $\Phi$, define a local logarithmic functor $\mathbb D^{\log}_{\Phi}$.
The action of $\pi_1^{\et}(\mathcal Y_K^\circ)$ on the corresponding period module is defined by a logarithmic Taylor expansion.

A key point is to show that the resulting local functor is independent of the choice of the Frobenius lifting.
For two Frobenius liftings, logarithmic Taylor isomorphisms give canonical identifications between the corresponding period modules, and these identifications are compatible with the Galois actions, filtrations, and Frobenius structures.
The local functors are therefore canonically identified on overlaps and can be glued to a global functor $\mathbb D^{\log}$.
We also compare this construction with Faltings' original $\mathbb D$-functor after removing the boundary divisor.
In particular, the logarithmic construction restricts canonically to the ordinary construction on $\mathcal U$.

Our main result is the following.

\begin{theorem}[\autoref{algebraicalize}]
With the notation above, the following statements hold.
\begin{enumerate}
\item
There is a functor $\mathbb D^{\log}$ from the category of torsion (resp. torsion-free) logarithmic Fontaine--Faltings modules over $(\mathcal Y,\mathcal Z)$ with Hodge--Tate weights in $[0,p-2]$ to the category of continuous representations of $\pi_1^{\et}(\mathcal Y_K^\circ)$ on finite $p$-primary torsion (resp. finite free) $\mathbb Z_p$-modules.
For every such logarithmic Fontaine--Faltings module $M$, there is a canonical isomorphism
\[
\left.
\mathbb D^{\log}(M)
\right|_{\mathcal U_K}
\xrightarrow{\sim}
\mathbb D(M|_{\mathcal U})
\]
of continuous representations of $\pi_1^{\et}(\mathcal U_K)$.

\item
If $Y$ is proper over $W$, let $Y_K$ and $Z_K$ be the generic fibers of $Y$ and $Z$, respectively, and let
$Y_K^\circ\coloneqq Y_K-Z_K$.
Then $\mathbb D^{\log}$ algebraizes to a functor from the category of torsion (resp. torsion-free) logarithmic Fontaine--Faltings modules over $(\mathcal Y,\mathcal Z)$ with Hodge--Tate weights in $[0,p-2]$ to the category of continuous representations of $\pi_1^{\et}(Y_K^\circ)$ on finite $p$-primary torsion (resp. finite free) $\mathbb Z_p$-modules.
\end{enumerate}
\end{theorem}

We further prove that the logarithmic $\mathbb D^{\log}$-functor is fully faithful on both the torsion and torsion-free categories; see \autoref{prop_FullyFaithful}.
The main ingredient is a logarithmic analogue of \cite[Lemma~2.5]{Fal89}.
More precisely, we establish a length estimate for a logarithmic Fontaine--Faltings module in terms of a representation mapping minimally to its associated $\mathbb D^{\log}$-representation.
This estimate allows us to adapt Faltings' argument to the logarithmic setting and deduce the full faithfulness of $\mathbb D^{\log}$.

The paper is organized as follows.
In Section~1, we recall the basic theory of Fontaine--Faltings modules, Faltings' period ring, and the $\mathbb D$-functor.
In Section~2, we introduce torsion and torsion-free logarithmic Fontaine--Faltings modules and construct the logarithmic period ring and the local $\mathbb D^{\log}$-functor.
We then prove the independence of the local construction from the choice of Frobenius lifting, glue the local functors to obtain the global $\mathbb D^{\log}$-functor, compare it with Faltings' $\mathbb D$-functor away from the boundary, and finally establish the algebraization and full faithfulness results.

\vspace{5mm}

\noindent\textbf{Acknowledgments.}
This paper was inspired by the suggestions of an anonymous referee on a paper co-authored by Raju Krishnamoorthyand the second and third authors. We express our heartfelt gratitude to the referee.
We are deeply indebted to Raju Krishnamoorthy for his valuable improvements to the initial version of this manuscript. His insights and contributions have been instrumental in shaping and advancing this work, and we hereby express our most sincere gratitude to him.
J.Y. and Z.L. acknowledge the financial support from the National Natural Science Foundation of China (Grant No. 12201595), the Fundamental Research Funds for the Central Universities, and the CAS Project for Young Scientists in Basic Research (Grant No. YSBR-032). K.Z. and J.Y. gratefully acknowledge the financial support from the Key Program (Grant No. 12331002) and the International Collaboration Fund (Grant No. W2441003) of the National Natural Science Foundation of China.

\section{Crystalline representations}
\label{section_FFM}

In this section, we review the theory of crystalline local systems.
We follow closely (both in notation and in exposition) the foundational work of Faltings \cite{Fal89}.
See also \cite[Section 1]{SYZ22} and \cite[Section 2]{LSZ13a}.
For a longer treatise with all of the details, see  \cite[Sections 2-4]{Tsu20}.
(This source proves in detail many of the necessary statements in commutative algebra to set up the period rings.
It also reproves several of Faltings' theorems using the recent $\mathbb{A}_{\text{inf}}$ technology.)  Throughout the paper, $k$ denotes a perfect field of characteristic $p>2$, $W\coloneqq W(k)$ the ring of Witt vectors over $k$, with fraction field $K=\operatorname{Frac}W$.

\begin{setup}
\label{setup:scheme_rigid1}
Let $Y$ be a smooth separated scheme over $W$ (not necessarily proper), whose generic and special fibers are both nonempty and geometrically connected.
Denote by ${\mathcal Y}$ the $p$-adic formal completion of $Y$ along the special fiber $Y_1$ and by $\mathcal Y_K$ the rigid-analytic space associated to ${\mathcal Y}$,
which is an open subset of $Y_K^{\rm an}$.\footnote{This is implied by \cite[Theorem A.3.1]{Con99}, and when $Y/W$ is proper, $\mathcal Y_K=Y_K^{\rm an}$.}
\end{setup}

Let $(V,\nabla)$ be a \emph{de Rham sheaf}
(i.e., a sheaf with an integrable connection)
over $Y_n\coloneqq Y\times_{\mathrm{Spec}(W)}\mathrm{Spec}(W_n)$.
In this paper, a filtration $\Fil$ on $(V,\nabla)$ will be called a
\emph{Hodge filtration of level in $[a,b]$}
if the following conditions hold:

\begin{itemize}
\item[-] $\Fil^i V$'s are locally split sub-sheaves of $V$, with
\[V=\Fil^aV\supset \Fil^{a+1}V \supset\cdots \supset \Fil^bV\supset \Fil^{b+1}V=0,\]
where every point has a neighborhood $U\subset Y_n$ such that each graded factor $\Fil^i V(U)/\Fil^{i+1} V(U)$ is a finite direct sum of $\mathcal O_{Y_n}(U)$-modules of the form $\mathcal O_{Y_n}(U)/p^e$.

\item[-] $\Fil$ satisfies Griffiths transversality with respect to the connection $\nabla$.
\end{itemize}
In this case, the triple $(V,\nabla,\Fil)$ is called a
\emph{filtered de Rham sheaf}.
Taking limits, one easily defines a (filtered) de Rham sheaf over the formal smooth scheme $\mathcal Y$.

\subsection{Fontaine--Faltings modules}
In this subsection, we recall the theory of Fontaine--Faltings modules, which provide integral lattice realizations of crystalline local systems defined on the rigid generic fiber of $p$-adic formal schemes. We first give the local definition over small affine formal bases, introduce Faltings' gluing functor to resolve the dependence on Frobenius liftings, and finally globalize the theory to arbitrary formal bases.

\subsubsection{Fontaine--Faltings modules over a small affine base.}
\label{sect:FF/small}

We first recall the notion of a Fontaine--Faltings module over a small affine scheme.
Assume $Y$ is a connected, \emph{small} affine scheme,
i.e., $Y$ is connected and if $Y=\mathrm{Spec}(R)$,
then $R/pR$ is nonzero and integral, and  there exists an \'etale map
\[W[T_1^{\pm1},T_2^{\pm1},\cdots, T_{d}^{\pm1}]\rightarrow R,\]
over $W$ (see \cite[p.~27]{Fal89}).
In general, a smooth affine scheme over $W$ is not always small but it can be covered by a system of small affine open subsets.
By the existence of the \'etale chart, there exists a $\Phi:\widehat{R}\rightarrow\widehat{R}$ which lifts the absolute Frobenius on $R/pR$, where $\widehat{R}$ is the $p$-adic completion of $R$.

A \emph{Fontaine--Faltings module} over the $p$-adic formal completion $\mathcal Y=\mathrm{Spf}(\widehat{R})$ of $Y$ with Hodge-Tate weights in $[a,b]$ is a quadruple $(V,\nabla,\Fil,\varphi)$, where
\begin{itemize}
\item[-] $(V,\nabla)$ is a de Rham $\widehat{R}$-module;
\item[-] $\Fil$ is a Hodge filtration on $(V,\nabla)$ of level in $[a,b]$;
\item[-] $\widetilde{V}$ is the quotient $\bigoplus\limits_{i=a}^b\Fil^i/\sim$. Here, $px\sim y$ for $x\in\Fil^iV$ where $y$ is the image of $x$ under the natural inclusion $\Fil^iV\hookrightarrow\Fil^{i-1}V$;
\item[-] $\varphi$ is an $\widehat{R}$-linear isomorphism \[\varphi:\widetilde{V}\otimes_{\Phi}\widehat{R} \longrightarrow V,\]
\item[-] The relative Frobenius $\varphi$ is horizontal with respect to the connections.
\end{itemize}
(The fact that $\varphi$ is an isomorphism is sometimes known as \emph{strong $p$-divisibility.}) A morphism between Fontaine--Faltings modules is a morphism between the underlying de Rham modules which is strict for the filtrations and commutes with the $\varphi$-structures. Denote by $\mathcal {MF}_{[a,b]}^{\nabla,\Phi}(\mathcal Y/W)$ the category of all Fontaine--Faltings modules over $\mathcal Y$ of Hodge-Tate weights in $[a,b]$. The $p$-primary torsion version of this definition was first written down in \cite[p.~30--31]{Fal89}, see also \cite[Section 2]{SYZ22} and \cite[Section 2]{LSZ13a}.

\subsubsection{The gluing functor}
 In the following, we recall the gluing functor of Faltings. In other words, up to a canonical equivalence of categories, if $b-a\leq p-1$, the category $\mathcal {MF}_{[a,b]}^{\nabla,\Phi}(\mathcal Y/W)$ does not depend on the choice of $\Phi$. More explicitly, the functor yielding an equivalence is given as follows. Let $\Psi$ be another lifting of the absolute Frobenius. For any filtered de Rham module $(V,\nabla,\Fil)$, Faltings~\cite[Theorem~2.3]{Fal89} shows that there is a canonical isomorphism by the Taylor formula
\[\alpha_{\Phi,\Psi}: \widetilde{V}\otimes_\Phi\widehat{R} \simeq \widetilde{V}\otimes_\Psi \widehat{R},\]
which is \emph{compatible} with the connection, satisfies the cocycle conditions and induces an equivalence of categories
\begin{equation}
\xymatrix@R=0mm{ \mathcal {MF}_{[a,b]}^{\nabla,\Psi}(\mathcal Y/W)\ar[r] & \mathcal {MF}_{[a,b]}^{\nabla,\Phi}(\mathcal Y/W).\\
(V,\nabla,\Fil,\varphi)\ar@{|->}[r] & (V,\nabla,\Fil,\varphi\circ\alpha_{\Phi,\Psi})\\}
\end{equation}

\subsubsection{Fontaine--Faltings modules over a global base.}
\label{sect:FF/gl}

In this section, we do not assume $Y$ is small,
but we maintain the assumption that $Y$ has a geometrically connected generic fiber.
Let $\{\mathcal U_i\}_{i\in I}$ be a connected small affine open covering of $\mathcal Y$
with $\Phi_i$ a lift of the absolute Frobenius on $\mathcal O_{\mathcal Y}(\mathcal U_i)\otimes_W k$ for any $i\in I$.
Recall that the category $\mathcal {MF}_{[a,b]}^{\nabla}({\mathcal Y}/W)$ is constructed by gluing the categories $\mathcal {MF}_{[a,b]}^{\nabla,\Phi_i}(\mathcal U_i/W)$.
Actually $\mathcal {MF}_{[a,b]}^{\nabla}(\mathcal Y/W)$ can be described more precisely as follows.
A Fontaine--Faltings module over ${\mathcal Y}$ of Hodge-Tate weights in $[a,b]$ is a tuple $(V,\nabla,\Fil,\{\varphi_i\}_{i\in I})$,
i.e., a filtered de Rham sheaf $(V,\nabla,\Fil)$ over $\mathcal Y$ together with
$\varphi_i: \widetilde{V}(\mathcal U_i)\otimes_{\Phi_i} \widehat{\mathcal O_{\mathcal Y}(\mathcal U_i)}\rightarrow V(\mathcal U_i)$
such that

\begin{itemize}
\item[-] $M_i\coloneqq (V(\mathcal U_i),\nabla,\Fil,\varphi_i)\in \mathcal {MF}_{[a,b]}^{\nabla,\Phi_i}(\mathcal U_i/W)$.

\item[-] For all $i,j\in I$, on the open intersection $\mathcal U_i\cap \mathcal U_j$, the Fontaine--Faltings modules $M_i\mid_{\mathcal U_{i}\cap \mathcal U_j}$ and $M_j\mid_{\mathcal U_{i}\cap \mathcal U_j}$ are associated to each other under the above equivalence of categories with respect to the two Frobenius liftings $\Phi_i$ and $\Phi_j$ on $\mathcal U_i\cap \mathcal U_j$.
\end{itemize}
Denote by $\mathcal {MF}_{[a,b]}^{\nabla}({\mathcal Y}/W)$ the category of all Fontaine--Faltings modules over ${\mathcal Y}$ of Hodge-Tate weights in $[a,b]$.

\subsection{Two auxiliary categories}
\label{sect:aux_cats}

For the proofs, Faltings requires two auxiliary abelian categories.
Again, suppose that $R/W$ is smooth (or a $p$-adic completion thereof).
Fix a Frobenius lift $\Phi\colon \hat{R}\rightarrow \hat{R}$.
(In our final application,
it will be convenient to assume that $\Phi$ is \'etale in characteristic $0$.)

Set $\MFbig{R}$ to be the following category.

\begin{itemize}
\item
An object is a $p$-primary torsion $R$-module $M$, a collection of $p$-primary torsion $R$-modules $F^i(M)$,
and a collection of $R$-linear maps $F^i(M)\rightarrow F^{i-1}(M)$, $F^i(M)\rightarrow M$, and $\varphi^i\colon F^i(M)\otimes _{R,\Phi} R\rightarrow M$
(or, equivalently, a $\Phi$-linear morphism $\varphi^i\colon F^i(M) \rightarrow M$)
satisfying the following conditions.
\begin{enumerate}

\item The composition $F^i(M)\rightarrow F^{i-1}(M)\rightarrow M$ is the map $F^i(M)\rightarrow M$.

\item The map $F^i(M)\rightarrow M$ is an isomorphism if $i\ll 0$.

\item The composition of $$F^i(M)\rightarrow F^{i-1}(M)\xrightarrow{\varphi^{i-1}} M$$
is $p\varphi^{i}$.
\end{enumerate}

\item Morphisms are $R$-linear maps that commute with all additional structure maps.
\end{itemize}

For an integer $r$, and an object $M\in\MFbig{R}$, we define its
\emph{truncation at degree $r$}, denoted by $M_{\leq r}$, to
have the same underlying $R$-module as $M$, with
\[
F^i(M_{\leq r})=
\begin{cases}
F^i(M), & i\leq r,\\
0, & i>r,
\end{cases}
\]
and with the same structure maps and divided Frobenius maps in degrees
$i\leq r$, while all structure maps in degrees $i>r$ are set equal to
zero. Then $M_{\leq r}$ is again an object of $\MFbig{R}$.
If $N\in\MFbig{R}$ satisfies $F^i(N)=0$ for all $i>r$, then for every
$M\in\MFbig{R}$ there is a natural identification
\[
\operatorname{Hom}_{\MFbig{R}}(N,M)
=
\operatorname{Hom}_{\MFbig{R}}(N,M_{\leq r}).
\]

We define $\MF{R}$ to be the full subcategory where the $M$, $F^i(M)$ are all finitely generated $p$-primary torsion $R$-modules,
where $F^i(M)=(0)$ for all $i\gg 0$, and where a strong $p$-divisibility condition holds:
$\{\varphi^i\}$ induces an isomorphism $\tilde{M}\otimes_{R,\Phi} R\rightarrow M$ (see \cite[p.~30--31]{Fal89} for the definition of $\tilde{M}$).
We emphasize: the categories $\MFbig{R}$ and $\MF{R}$ are \emph{not} independent of the choice of Frobenius lift.
Nonetheless, they will be useful to us, especially in the context of small open affines.

We recall some basic properties of $\MF{R}$.
\begin{itemize}
\item
Let $(M,F^i(M),\{\varphi^i\})$ be an object of $\MF{R}$.
Then the maps $F^{i+1}(M)\rightarrow F^{i}(M)\rightarrow M$ are all injections onto direct summands (as $R$-modules), see \cite[Theorem 2.1 (i)]{Fal89}.

\item
 The underlying module of an object $M$ is locally  isomorphic to a direct sum of $R$-modules $R/p^eR$, see \cite[Theorem 2.1 (ii)]{Fal89};

\item
The category $\MF{R}$ is an abelian category and moreover, any map in $\MF{R}$ is \emph{strict} for the filtrations, see \cite[Corollary, Page 33]{Fal89};

\item
Let $(V,\nabla,\Fil,\varphi)\in \mathcal{MF}_{[a,b]}^{\nabla}(\mathcal Y/W)$. Let $\varphi^i$ denote the composition of $\varphi$ with the canonical map $\Fil^i V\to \widetilde{V}$.
Then the triple $(V/p^n,\Fil,\{\varphi^i\})$ is an object of $\MF{R}$ for every integer $n\geq 1$. This follows directly from the definition of a Fontaine--Faltings module.

\end{itemize}
Note that once we restrict the underlying modules to be finitely generated, then by $p$-primary torsionness they are automatically $p$-adically complete. (In particular, there is an equivalence of categories between finitely generated $p$-primary torsion modules on $R$ and finitely generated $p$-primary torsion modules on $\hat{R}$.)

\subsection{The $\mathbb D$-functor of Fontaine-Laffaille-Faltings}
\label{section FDF}

We recall the Fontaine-Laffaille-Faltings $\mathbb D$-functor in this subsection.
Before doing this, we recall the period ring $B^+(R)$ from \cite[p.~27--28]{Fal89} associated to a smooth (or formally smooth) $W$-algebra $R$.

\subsubsection{The period ring}\label{section Period-ring}
Let $\operatorname{Frac}(R)$ denote the fraction field of $R$, and fix an algebraic closure $\overline{\operatorname{Frac}(R)}$ of $\operatorname{Frac}(R)$.
Denote by $\overline{R}$ the maximal extension of $R$ that is \'etale in characteristic $0$.
More precisely, if $R$ is geometrically integral, then $\overline{R}$ is the integral closure of $R$ in the maximal subextension of $\operatorname{Frac}(R)$ contained in $\overline{\operatorname{Frac}(R)}$ for which the normalization of $R[1/p]$ in this subextension is unramified over $R[1/p]$.
Consider the ring
$$S = (\overline{R})^{\flat}\coloneqq  \varprojlim(\overline{R}/p\overline{R}),$$
where the limit is over a projective system of rings indexed by the number $n\geq0$ with transition maps given by Frobenius.
There is a natural surjective ring homomorphism
$$\theta\colon \Ainf{\bar{R}}\coloneqq W(S)\longrightarrow \widehat{\overline{R}}$$
from the Witt vectors of $S$ to the $p$-adic completion of $\overline{R}$ with $I\coloneqq \ker \theta$.
Denote by
\[
B^+(R)\coloneqq D_{I}(\Ainf{\bar{R}})^{\wedge}
\]
the completion of the divided power hull of $\Ainf{\bar{R}}$ with respect to the topology induced by the divided power ideals $\bigl(I+p\Ainf{\bar{R}}\bigr)^{[n]}$ for $n\ge 1$.
From the construction, the Galois group $\mathrm{Gal}\left(\overline{R}[\frac1p]/R[\frac1p]\right)$\footnote{
This group was originally denoted by $\mathrm{Gal}\left(\overline{R}/R\right)$ in \cite{Fal89}.
}
operates continuously\footnote{
The continuity was asserted in \cite[p.~28]{Fal89}, whose proof is the same as the case considered by Fontaine \cite{Fon82,Fon82Coh}.
}
on $B^+(R)$.
We note that $B^+(R)$ is endowed with natural Frobenius and filtration structures, and the filtration structure is defined via divided power ideals.

More precisely, the Frobenius on $S=\overline{R}^{\flat}$ induces a canonical Frobenius on $\Ainf{\bar{R}}$,
which naturally extends to a Frobenius on both of the above rings: $D_I(W(S))$ and $B^+(R)$.
The filtration on $D_I(W(S))$ is defined by
\[
\Fil^n D_I(W(S)) \coloneqq I^{[n]}\cap \ker\left( D_I(W(S)) \xrightarrow{\mathrm{Frob}} D_I(W(S)) \longrightarrow D_I(W(S))/p^n \right). \]
Notice that for $n<p$ this agrees with the divided power filtration, namely
\[
\Fil^n D_I(W(S))=I^{[n]}.
\]
For $n\geq p$, $I^{[n]}/\Fil^n D_I(W(S))$ is $p$-torsion. Moreover, the filtration $\Fil^\bullet$ is strongly $p$-divisible by construction.
The filtration on $B^+(R)$ is defined to be the closure of that on $D_I(W(S))$.

We now give a crucial definition, which may be found on the bottom two lines of \cite[p.~35]{Fal89}.
This definition will be used to provide a functor from $\mathcal {MF}_{[a,b]}^{\nabla}(\text{Spf}(R)/W)$ to the category of continuous $p$-adic representations of $\text{Gal}(\bar R[\frac1p]/R[\frac1p])$.

\begin{definition}
\label{def:D}
Let $A$ be a smooth $W$-algebra or a localization/$p$-adic completion thereof. We set

\begin{equation}
\label{equation:D}
D(A)=B^+(A)[1/p]/B^+(A),
\end{equation}
which is equipped with the natural $\varphi$-structure and ``filtration'' (which is \textbf{not} by sub-objects).
To define the filtration, note the following:
$D(A)\cong\varinjlim_n B^+(A)/p^nB^+(A)$,
where the transition maps are multiplication by $p$.
The filtration on $B^+(A)$ induces one on $B^+(A)/p^nB^+(A)$,
which gives us the filtration on $D(A)$.
\end{definition}

Note that $D(A)$ is not an object of $\MFbig{A}$ as it has no natural $A$ structure.

\subsubsection{The local $\mathbb{D}$-functor}
Let $Y$, $\mathcal Y$, $R$ and $\Phi$ be given as in \autoref{sect:FF/small}.
Suppose $\Phi$ happens to be \'etale in characteristic $0$
(which may be guaranteed if $Y$ is small).
Let $M=(V,\nabla,\Fil,\varphi)$ be an object in $\mathcal {MF}_{[0,p-2]}^{\nabla,\Phi}(\mathcal Y/W)$.
We will now explain roughly how to ``evaluate'' $M$ on $B^{+}(\widehat{R})$.
This is analogous to the evaluation of a crystal,
but is more subtle here because $\overline{\widehat{R}}/W$ is far from being smooth.

There is a natural map
$\widehat{R}\rightarrow \widehat{\overline{\widehat{R}}}\cong B^+({\widehat{R}})/\Fil^1$.
By the formal smoothness of $\widehat{R}/W$,
we may lift this to a $W$-algebra homomorphism:
$$\widehat{R}\rightarrow B^+({\widehat{R}}).$$
Faltings shows that there is a particular such lift $\kappa_\Phi:\widehat{R}\rightarrow B^+(\widehat{R})$
which respects Frobenius lifts \cite[p.~36]{Fal89}.
Thus the following diagram commutes
\begin{equation}\label{equ:A3}
\xymatrix@C=2cm{
\widehat{R} \ar[r]^{\kappa_\Phi} \ar[d]^{\Phi}
& B^+(\widehat{R}) \ar[d]^{\Phi_B}\\
\widehat{R} \ar[r]^{\kappa_{\Phi}}
& B^+(\widehat{R}).\\
}
\end{equation}
Here $\Phi_B$ is the aforementioned Frobenius on $B^+(\widehat{R})$.
(This crucially uses an \'etale chart $W[T_1^{\pm 1},\dots,T_d^{\pm 1}]\rightarrow R$.)

Now, $F^1$ admits divided powers (by construction).
As the connection $\nabla$ satisfies Griffiths transversality,
it will follow that the following filtered $B^+({\widehat{R}})$-module:
$$\mathbb P^{\cris}_{\Phi}(M)\coloneqq V\otimes _{\widehat{R},\kappa_{\Phi}}B^+({\widehat{R}})$$
is \emph{independent} of the choice of the lift $\kappa_{\Phi}$.
Moreover, $\mathbb P^{\cris}_{\Phi}(M)$ has a Frobenius structure induced by the tensor of the Frobenius structure on each of the factors (recall,
$\widehat{R}$ had an \emph{\'etale} Frobenius structure $\Phi$).
Finally, in the case that $\Phi$ is \'etale, $D(R)$ is in fact in $\MFbig{R}$.

Now, $\kappa_{\Phi}$ cannot be made equivariant for the action of
$\pi_1^{\et}(\mathcal Y_K)=\mathrm{Gal}(\overline{\widehat{R}}[\frac1p]/\widehat{R}[\frac1p])$.\footnote{
Here, by \'etale fundamental group, we mean the fundamental group corresponding to the category of \emph{finite} \'etale covers.
}
Following Faltings, we can however find a horizontal action of $\pi_1^{\et}(\mathcal Y_K)$ on $\mathbb P^{\cris}_{\Phi}(M)$ via the connection $\nabla$ on $V$, which commutes with the $\varphi$'s. For $\sigma\in \pi_1^{\et}(\mathcal Y_{K})$, $v\in V$ and $r\in B^+(\widehat{R})$, the action is given by
\begin{equation} \label{equ:Galois}
    \sigma(v\otimes r)\coloneqq  \left(\sum_I\nabla(\partial)^I(v)\otimes \Delta^I/I!\right)\cdot \sigma(r),
\end{equation}
where, if we write $I=(i_1,\dots,i_d)\in\mathbb N^d$, then
\[
\Delta^I \coloneqq \prod_{k=1}^d (\sigma\circ \kappa_{\Phi}(T_k)-\kappa_{\Phi}(T_k))^{i_k},
\quad
\nabla(\partial)^I\coloneqq \prod_{k=1}^d \nabla\left(\frac{\partial}{\partial T_k}\right)^{i_k}.
\]
When $\Phi$ denotes the standard Frobenius lifting sending $T_i$ to $T_i^p$, \eqref{equ:Galois} coincides with the action given in the first paragraph of \cite[p.~37]{Fal89}.

Faltings also showed that this Galois module does not depend on the choice of the Frobenius lifting $\Phi$.
That is, for any other Frobenius lifting $\Psi$,
there is a canonical isomorphism between $\pi_1^{\et}(\mathcal Y_K)$-modules
\begin{equation}
\label{equ_cano_iso_phi_psi}
\mathbb P^{\cris}_{\Phi}(M) \xrightarrow{\sim} \mathbb P^{\cris}_{\Psi}(M).
\end{equation}
(See \eqref{equ:A3} for the notation.)
In particular, when we write $\Pcris{M}$ we consider it as a $\pi_1^{\et}(\mathcal Y_K)$-module.

We now define the functor $\mathbb D$.
In the case that $V$ is killed by some power of $p$,
denote
\[\mathbb{D}(M)=\mathrm{Hom}(\Pcris{M},D(\widehat{R}))\]
where the homomorphisms are $B^+(\widehat R)$-linear and respect filtrations and the $\varphi$-structure.
In particular, when $\Phi$ is \'etale in characteristic $0$,
this is $\mathrm{Hom}$ in $\MFbig{R}$.

In the case that $V$ is $p$-torsion-free, denote
\[\mathbb{D}(M)=\varprojlim_n\mathrm{Hom}(\Pcris{M}/p^n,D(\widehat{R})).\]
(Note that $\Pcris{M}/p^n\cong \Pcris{M/p^n}$.) The fundamental group $\pi_1^{\et}(\mathcal Y_K)$ acts continuously on $\Pcris{M}$ and $D(\widehat{R})$, and so also on $\mathbb{D}(M)$.

\subsubsection{The global $\mathbb{D}$-functor}
Let $Y$ be a proper smooth scheme over $W$ with geometrically connected generic fiber.
Denote by $\mathcal Y$ the formal completion of $Y$.
Let $I$ index the collection of pairs $(\mathcal U_i,\Phi_i)$,
where $\mathcal U_i$ is a small open affine of ${\mathcal Y}$ and $\Phi_i$ is a Frobenius lift \'etale in characteristic $0$.
(Again, if $\mathcal U_i$ is a small open affine,
then there exists a Frobenius lift \'etale in characteristic $0$
by raising the $T_i$ coordinates to the $p^{\text{th}}$ power,
see \cite[p.~35]{Fal89}.)
The local $\mathbb D$-functors over each $\mathcal U_i$ are compatible on overlaps,
due to the existence of the canonical isomorphism in \eqref{equ_cano_iso_phi_psi}.
Therefore, given a Fontaine--Faltings module over ${\mathcal Y}$,
we obtain a compatible system of \'etale sheaves on $\mathcal U_{i,K}$
(the generic fiber of $\mathcal U_i$).
Then it follows from \cite[Theorem 3.1]{BG98} that these local systems glue to a $\mathbb Z_p$-local system on $\mathcal Y_K$.
More precisely, we use faithfully flat descent for rigid spaces \cite[Theorem 4.2.9]{Con06}
(which takes \cite[Theorem 3.1]{BG98} as crucial input via a relative (analytic) Proj construction).\footnote{
Here, the covers are all \emph{finite}, hence one could also use a relative Spec construction in conjunction with \cite[Theorem 3.1]{BG98}.
}
An \'etale $\mathbb Z_p$-local system on $Y_K$ amounts to a system of finite \'etale covers of $Y_K$.
A very special case of the aforementioned faithfully flat descent shows that to specify such a system of finite \'etale covers on $Y_K$,
it is precisely equivalent to construct such covers on a \emph{finite} admissible cover of $Y_K$ which satisfy a cocycle/gluing condition.
(The point is that given a finite admissible cover $(\mathcal U_i)_K$ of $Y_K$,
the disjoint union $\bigsqcup (\mathcal{U}_i)_K\rightarrow Y_K$ will tautologically admit local fpqc quasi-sections,
as in \cite[Definition 4.2.1]{Con06}.)
We therefore obtain a functor
\begin{equation}
\mathbb D\colon \mathcal {MF}_{[0,p-2]}^{\nabla}({\mathcal Y}/W)\rightarrow \text{Loc}_{\mathbb Z_p}(\mathcal Y_K).
\end{equation}

Since $Y$ is proper, $\mathcal Y_K=Y^{\mathrm{an}}_K$.
By rigid GAGA, see \cite[Theorem 3.1]{Lut93},
the local system we obtain over $\mathcal Y_K$ is algebraic,
i.e., one obtains a $\mathbb Z_p$-local system on $Y_K$.\footnote{
Faltings' original idea is to use formal GAGA from EGA3 (\cite[Th\'eor\`eme 5.1.4]{Gro61EGA3}). The \'etale sheaves on $\mathcal U_{i,K}$ can be expressed in terms of finite \'etale coverings.
In \cite[p.~42]{Fal89}, he wanted to extend these local \'etale coverings into finite coverings over local formal schemes $\mathcal U_i$ and to glue them into a global finite covering of $\mathcal Y$.
It would then follow from formal GAGA that this finite
cover, \'etale on the generic fiber, of $\mathcal Y$ algebraizes to a finite cover of $Y$,
which is \'etale over the generic fiber $Y_K$. Unfortunately, Faltings did not write down the details;
the necessary formal gluing lemma was later carefully proven in \cite[Theorem A3]{Tsu96}.
We take a different approach, using rigid gluing of finite \'etale covers.
}
This has the following upshot:
there is a functor, by abusing notation we still denote it by $\mathbb D$,
\begin{equation}
\mathbb D\colon \mathcal {MF}_{[0,p-2]}^{\nabla}({\mathcal Y}/W)\rightarrow \text{Loc}_{\mathbb Z_p}(Y_K),
\end{equation}
from the category of Fontaine--Faltings modules (with Hodge--Tate weights in $[0,p-2]$) to the category of finite-dimensional lisse $\mathbb Z_p$ sheaves on $Y_K$.
Since $Y_K$ is connected, picking a base point,
this is equivalent to the category $\mathrm{Rep}_{\mathbb Z_p}(\pi_1^{\et}(Y_K))$ of continuous $\mathbb Z_p$-representations of $\pi_1^{\et}(Y_K)$.
It is a fundamental result of Faltings that this functor is fully faithful;
a lisse $\mathbb Z_p$ sheaf in the essential image of $\mathbb{D}$ is called \emph{crystalline (with Hodge-Tate weights in $[0,p-2]$).}

\section{Logarithmic crystalline representations}
\label{section:log_FFM}

Faltings claimed in \cite[i) p.43]{Fal89} that the theory of crystalline representations extends to the logarithmic context.
However, it seems as though the details of this construction have never appeared in the literature.
In this section, we explicitly write out this construction.
One may also find the definition of a logarithmic Fontaine--Faltings module in \cite{LSZ13a}.

\subsection{Logarithmic Fontaine--Faltings modules}
In this subsection, we first define logarithmic Fontaine--Faltings modules locally on logarithmically small affine pairs. We then prove the logarithmic gluing theorem and use it to define logarithmic Fontaine--Faltings modules over a general pair $(\mathcal Y,\mathcal Z)$.

\subsubsection{Local logarithmic Fontaine--Faltings modules}
\begin{setup}
\label{setup:Aff_log_FF}
Let $s,d$ be non-negative integers with $s\le d$. Let $Y=\operatorname{Spec}R$ be an affine smooth connected $W$-scheme with $R/pR$ nonzero and integral, equipped with an \'etale $W$-morphism
\[
Y\longrightarrow \operatorname{Spec} W[T_1,T_2,\dots, T_s, T_{s+1}^{\pm 1},\dots, T_{d}^{\pm 1}].
\]
Define $Z\subset Y$ as the divisor on $Y$ cut out by $T_1\cdots T_s=0$, and set $U\coloneqq Y- Z$ to be its open complement. Then $U$ is a small affine scheme. In this setting, we call the pair $(Y,Z)$ \emph{logarithmically small}.
We construct the objects $Z_K$, $\mathcal Z$, $\mathcal Z_K$, $U_K$, $\mathcal U$, and $\mathcal U_K$ exactly by analogy with the corresponding objects attached to $Y$ in \autoref{setup:scheme_rigid1}. Write $\mathcal Y^\circ_K\coloneqq \mathcal Y_K\setminus\mathcal Z_K$. Let $\widehat{R}$ denote the $p$-adic completion of $R$, so that $\mathcal Y=\operatorname{Spf}(\widehat{R})$. Let $\Phi\colon \widehat{R}\to\widehat{R}$ be a lifting of the absolute Frobenius on $R/pR$ satisfying:
\begin{enumerate}
    \item $\Phi(T_i)=w_i T_i^p$ for some $w_i\in \widehat{R}^\times$, for $1\le i\le s$;
    \item $\Phi$ induces a finite \'etale endomorphism of $\widehat{R}[1/(pT_1\cdots T_s)]$.
\end{enumerate}
\end{setup}

Such a lifting $\Phi$ always exists: the ring $W\langle T_1,\dots,T_s,T_{s+1}^{\pm1},\dots,T_d^{\pm1}\rangle$ carries an explicit Frobenius lifting sending each $T_i\mapsto T_i^p$ for $1\le i\le d$. Since $\widehat{R}$ is \'etale over this ring, the Frobenius lifts uniquely to $\widehat{R}$. The resulting lift automatically satisfies conditions (1) and (2).

We note that since $\Phi$ lifts the absolute Frobenius, $w_i\equiv 1\pmod{p}$. One can therefore write $w_i=1+pu_i$ for some $u_i\in \widehat{R}$. Then $\Phi$ induces a morphism of logarithmic differential modules
\[\mathrm{d} \Phi:\Omega^1_{\widehat{R}/W}(\log \mathcal Z)\otimes_{\Phi}\widehat{R}\longrightarrow \Omega^1_{\widehat{R}/W}(\log \mathcal Z),\]
whose image is contained in $p\Omega^1_{\widehat{R}/W}(\log \mathcal Z)$.

One defines a logarithmic Fontaine--Faltings module over $(\mathcal Y,\mathcal Z)$ in the same way as in the non-logarithmic case, simply using differentials with logarithmic poles instead of regular differentials.
More explicitly, a \emph{torsion logarithmic Fontaine--Faltings module} over the $p$-adic formal completion $(\mathcal Y,\mathcal Z)$ of $(Y,Z)$ with Hodge-Tate weights in $[a,b]$ is a quadruple $(V,\nabla,\Fil,\varphi)$, where
\begin{itemize}
\item
$(V,\nabla)$ is a finitely generated $p$-primary torsion de Rham $\widehat{R}$-module with logarithmic poles along $T_1\cdots T_s=0$;

\item
$\Fil$ is a Hodge filtration on $(V,\nabla)$ of level in $[a,b]$ as in the text after Setup \ref{setup:scheme_rigid1};

\item $\widetilde{V}$ is the quotient
$\bigoplus\limits_{i=a}^b\Fil^iV/\sim$
with $px\sim y$ for $x\in\Fil^iV$ with $y$ being the image of $x$ under the natural inclusion $\Fil^iV\hookrightarrow\Fil^{i-1}V$;
In other words, if we denote by $[v]_i$ the image of $v\in \Fil^iV$ in $\widetilde{V}$ under the natural morphism $\Fil^iV\rightarrow \widetilde{V}$,
then $[v]_{i-1} = p\cdot [v]_i$. $\widetilde{\nabla}$ is the integrable $p$-connection on $\widetilde{V}$ induced by $\nabla$.

\item
$\varphi$ is an $\widehat{R}$-linear isomorphism
\[\varphi:\widetilde{V}\otimes_{\Phi}\widehat{R} \longrightarrow V,\]

\item $\varphi$ is horizontal with respect to the connections, i.e., the diagram
\begin{equation*}
    \xymatrix{
    \widetilde{V}\otimes_{\Phi}\widehat{R} \ar[r]^{\varphi}\ar[d]_{\nabla'} & V\ar[d]^{\nabla}\\
    (\widetilde{V}\otimes_{\Phi}\widehat{R})\otimes \Omega^1_{\widehat{R}/W}(\log \mathcal Z)\ar[r]_{\qquad\varphi\otimes \mathrm{id}}& V\otimes \Omega^1_{\widehat{R}/W}(\log \mathcal Z)
    }
\end{equation*}
commutes. Here $\nabla'$ denotes the connection on $\widetilde{V}\otimes_{\Phi}\widehat{R}$ defined by
$$\nabla'(e\otimes f)=(e\otimes 1)\otimes \mathrm{d} f+f(\mathrm{id}_{\widetilde{V}\otimes_{\Phi}\widehat{R}}\otimes \mathrm{d} \Phi/p)(\widetilde{\nabla}(e)\otimes_{\Phi} 1),\quad e\in \widetilde{V},\quad f\in \widehat{R}.$$
\end{itemize}
Denote by $\mathcal {MF}_{[a,b]}^{\nabla,\Phi}((\mathcal Y,\mathcal Z)/W)$ the category of torsion logarithmic Fontaine--Faltings modules over $(\mathcal Y,\mathcal Z)$ with Hodge-Tate weights in $[a,b]$.

We next define the torsion-free version by passing to inverse limits with respect to the $p$-power torsion quotients. A \emph{torsion-free logarithmic
Fontaine--Faltings module} over $(\mathcal Y,\mathcal Z)$ with
Hodge--Tate weights in $[a,b]$ is a quadruple $(V,\nabla,\Fil,\varphi)$ such that:

\begin{itemize}
\item
$V$ is a finite projective $\widehat R$-module, and
\[
\nabla:V\longrightarrow
V\otimes_{\widehat R}
\Omega^1_{\widehat R/W}(\log\mathcal Z)
\]
is an integrable logarithmic connection;

\item
$\Fil$ is a decreasing and locally split filtration on $V$ of level in $[a,b]$, satisfying Griffiths transversality
\[
\nabla(\Fil^iV)
\subset
\Fil^{i-1}V\otimes_{\widehat R}
\Omega^1_{\widehat R/W}(\log\mathcal Z);
\]

\item
$\widetilde V$, $\widetilde{\nabla}$, and the induced connection on $\widetilde V\otimes_{\Phi}\widehat R$ are defined in the same way as in the torsion case above, and
$\varphi: \widetilde V\otimes_{\Phi}\widehat R \xrightarrow{\sim} V$
is an $\widehat R$-linear isomorphism which is horizontal with respect to this induced connection and $\nabla$.
\end{itemize}
Denote by $\mathcal {MF}_{[a,b]}^{\nabla,\mathrm{tf},\Phi}((\mathcal Y,\mathcal Z)/W)$ the category of torsion-free  logarithmic Fontaine--Faltings modules over $(\mathcal Y,\mathcal Z)$ with Hodge-Tate weights in $[a,b]$. Notice that for any object $M\in \mathcal {MF}_{[a,b]}^{\nabla,\mathrm{tf},\Phi}((\mathcal Y,\mathcal Z)/W)$ and any positive integer $n$, the reduction $M/p^n M$ is a torsion logarithmic Fontaine--Faltings module.

We say that a torsion or torsion-free logarithmic Fontaine--Faltings module $M=(V,\nabla,\Fil,\varphi)$ with Hodge--Tate weights in $[a,b]$ \emph{admits an adapted basis} if
there exist elements $m_1,\ldots,m_h\in V$ and integers
$e_1,\ldots,e_h\in[a,b]$ such that, for every $j$,
\[\Fil^j V = \bigoplus_{e_i\geq j}\widehat R\,m_i. \]
In this case, $\{m_1,\ldots,m_h\}$ is called an \emph{adapted basis} of $M$. By \cite[Theorem~2.1]{Fal89}, every torsion logarithmic Fontaine--Faltings module admits an adapted basis Zariski locally. The same is true for torsion-free logarithmic Fontaine--Faltings modules, since the underlying module is finite projective and the filtration is locally split.

\subsubsection{The logarithmic gluing functor}

We now show that the local categories defined above are canonically independent of the choice of Frobenius lift. The following elementary identity will be used in the proof of the cocycle condition.

\begin{lemma}\label{lem_coefficient}
    Let $f_n(X)=n!\tbinom{X}{n}=X(X-1)(X-2)\cdots(X-n+1)\in \mathbb Q[X]$ for every positive integer $n$ and let $f_0=1$. $\{f_n\}_{n\in \mathbb N}$ forms a basis of $\mathbb Q[X]$ over $\mathbb Q$. Set $a_{mn}^k\in \mathbb Q$ such that
    \begin{equation*}
        f_n(X)f_m(X)=\sum_{k=0}^{m+n}a_{mn}^k f_k(X).
    \end{equation*}
    Then
    \[\sum_{m,n\geq 0}a_{mn}^k\frac{(X-1)^m}{m!}\frac{(Y-1)^n}{n!}=\frac{(XY-1)^k}{k!}\]
    in $\mathbb Q[X,Y]$ for all $k\in \mathbb N$.
\end{lemma}
\begin{proof}
    In the ring $\mathbb Q[[X-1,Y-1,t]]$, one has
    \begin{equation*}
    \begin{split}
       \sum_{k=0}^{\infty} f_k(t)\sum_{m,n\geq 0}a_{mn}^k\frac{(X-1)^m}{m!}\frac{(Y-1)^n}{n!}
       &=\sum_{m,n\geq 0}f_m(t)f_n(t)\frac{(X-1)^m}{m!}\frac{(Y-1)^n}{n!}\\
       &=\left(\sum_{m=0}^{\infty}f_m(t)\frac{(X-1)^m}{m!}\right)\left(\sum_{n=0}^{\infty}f_n(t)\frac{(Y-1)^n}{n!} \right) \\
       &=X^tY^t=(XY)^t\\
       &=\sum_{k=0}^{\infty}f_k(t)\frac{(XY-1)^k}{k!}.
    \end{split}
    \end{equation*}
By comparing the coefficients of $f_k(t)$ in the equation,  we have
\[\sum_{m,n\geq 0}a_{mn}^k\frac{(X-1)^m}{m!}\frac{(Y-1)^n}{n!}=\frac{(XY-1)^k}{k!}.\qedhere\]
\end{proof}

\begin{theorem}[Logarithmic version of Faltings' gluing theorem]
\label{thm_log_version_gluing}
With notation as in Setup~\ref{setup:Aff_log_FF}, assume that
$0\leq b-a\leq p-1$ and $p>2$.
Let $\Phi$ and $\Psi$ be two Frobenius lifts satisfying the conditions
in Setup~\ref{setup:Aff_log_FF}. Then there are equivalences of categories
\[
\mathcal {MF}_{[a,b]}^{\nabla,\Psi}
((\mathcal Y,\mathcal Z)/W)
\xrightarrow{\sim}
\mathcal {MF}_{[a,b]}^{\nabla,\Phi}
((\mathcal Y,\mathcal Z)/W)
\]
and
\[
\mathcal {MF}_{[a,b]}^{\nabla,\mathrm{tf},\Psi}
((\mathcal Y,\mathcal Z)/W)
\xrightarrow{\sim}
\mathcal {MF}_{[a,b]}^{\nabla,\mathrm{tf},\Phi}
((\mathcal Y,\mathcal Z)/W).
\]
For any third Frobenius lift satisfying the same conditions, these
equivalences satisfy the cocycle condition. Consequently, up to
canonical equivalence, both $\mathcal {MF}_{[a,b]}^{\nabla,\Phi}
((\mathcal Y,\mathcal Z)/W)$ and
$
\mathcal {MF}_{[a,b]}^{\nabla,\mathrm{tf},\Phi}
((\mathcal Y,\mathcal Z)/W)
$
are independent of the choice of the Frobenius lift $\Phi$.
\end{theorem}

\begin{proof}
The torsion-free case follows from the torsion case by reducing modulo $p^n$ for every $n\geq1$ and then passing to the inverse limit.
Thus it suffices to prove the theorem for torsion logarithmic Fontaine--Faltings modules.
For any two liftings $\Phi$ and $\Psi$, we will define a canonical isomorphism
\begin{equation}
\alpha_{\Phi,\Psi}: \widetilde{V}\otimes_\Phi\widehat{R} \simeq \widetilde{V}\otimes_\Psi \widehat{R},
\end{equation}
by the following rule (analogous to the non-logarithmic case given in \cite[p.~34]{Fal89}).
Without loss of generality, we assume $a=0$ and $b=p-1$; otherwise shift the numbering.
Let $T_1,\cdots,T_d$ be local coordinates such that our objects have logarithmic poles along $T_1\cdots T_s=0$.
The Frobenius lifts $\Phi$ and $\Psi$ have been required,
by the conditions specified in Setup \ref{setup:Aff_log_FF},
to respect the divisor,
thus for any $1\leq i\leq s$, there exist $u_i,v_i\in\widehat{R}$ such that $\Phi(T_i) = T_i^p(1+pu_i)$ and $\Psi(T_i) = T_i^p(1+pv_i)$.
In particular
\[\Phi(T_i)/\Psi(T_i) = (1+pu_i)(1+pv_i)^{-1} \in 1+p\widehat{R}.\]
Hence, for any multi-index $I=(i_1,\cdots,i_d)$ we get an endomorphism of $V$
\[
\nabla(\partial)^{\{I\}} \coloneqq
\biggl(\prod_{l=1}^s\prod_{k=0}^{i_l-1}\left(\nabla\left(T_l\frac{\partial}{\partial T_l}\right)-k\right)\biggr)
\cdot
\biggl(\prod_{l=s+1}^d \nabla\left(\frac{\partial}{\partial T_l}\right)^{i_l}\biggr).
\]
Using the notation \(f_n\) and \(a^{k}_{mn}\) introduced in \autoref{lem_coefficient}, for any two multi-indices $I$ and $J$, we have
\[
\begin{aligned}
\nabla(\partial)^{\{I\}}\circ \nabla(\partial)^{\{J\}}
&=
\prod_{l=1}^s
\left[
f_{i_l}\left(\nabla\left(T_l\frac{\partial}{\partial T_l}\right)\right)
f_{j_l}\left(\nabla\left(T_l\frac{\partial}{\partial T_l}\right)\right)
\right]
\prod_{l=s+1}^d
\nabla\left(\frac{\partial}{\partial T_l}\right)^{i_l+j_l} \\
&=
\prod_{l=1}^s
\left[
\sum_{k_l=0}^{i_l+j_l}
a_{i_lj_l}^{k_l}
f_{k_l}\left(\nabla\left(T_l\frac{\partial}{\partial T_l}\right)\right)
\right]
\prod_{l=s+1}^d
\nabla\left(\frac{\partial}{\partial T_l}\right)^{i_l+j_l} \\
&=
\sum_{k_1,\ldots,k_s \in \mathbb N}
\left(\prod_{l=1}^s a_{i_lj_l}^{k_l}\right)
\biggl(\prod_{l=1}^s\prod_{r=0}^{k_l-1}
\left(\nabla\left(T_l\frac{\partial}{\partial T_l}\right)-r\right)\biggr)
\biggl(\prod_{l=s+1}^d
\nabla\left(\frac{\partial}{\partial T_l}\right)^{i_l+j_l}\biggr).
\end{aligned}
\]
Set \[
a_{IJ}^{K}\coloneqq
\begin{cases}
\displaystyle \prod_{l=1}^{s} a_{i_l j_l}^{k_l},
  &\text{if } k_l = i_l + j_l\text{ for any } l\in\{s+1,\dots,d\},\\[12pt]
0, & \text{otherwise}.
\end{cases}
\]
Thus
\begin{equation}\label{equ_connection}
    \nabla(\partial)^{\{I\}}\circ \nabla(\partial)^{\{J\}}=\sum_{K\in \mathbb N^d} a_{IJ}^K\nabla(\partial)^{\{K\}}.
\end{equation}
Similarly, given a multi-index $I$ as above, we denote
\[
z_{\Phi,\Psi}^I\coloneqq \biggl(\prod_{j=1}^s \biggl(\Phi(T_j)/\Psi(T_j)-1\biggr)^{i_j}\biggr)\cdot \biggl(\prod_{j=s+1}^d \bigl(\Phi(T_j)-\Psi(T_j)\bigr)^{i_j}\biggr).
\]
Note that this is divisible by $p^{|I|}$, where $|I|=i_1+\cdots+i_d$ is the order of $I$.
Finally, set $I!=i_1!\cdots i_d!$.

Choose $m\in \Fil^i(V)$, via the natural map $\Fil^i(V) \rightarrow \widetilde{V}$, $m$ defines an element $[m]_i\in \widetilde{V}$.
Then we have a well-defined and $\widehat{R}$-linear map $\alpha_{\Phi,\Psi}$ given by
\begin{equation}
\label{LogConn}
\alpha_{\Phi,\Psi}([m]_i\otimes_\Phi 1) = \sum_{I} \left[\nabla(\partial)^{\{I\}}(m)\right]_{\max\{0,i-|I|\}} \otimes_{\Psi} \frac {z_{\Phi,\Psi}^I}{I!\cdot p^{\min\{i,|I|\}}}.
\end{equation}
\eqref{LogConn} is well-defined. Indeed, since $z_{\Phi,\Psi}^I \in p^{|I|}$, the term $z_{\Phi,\Psi}^I/(I!\cdot p^{\min\{i,|I|\}})$ tend to zero as $|I|\to\infty$. Thus the series converges and \eqref{LogConn} is well defined.
We note that, by Griffiths' transversality, $\nabla(\partial)^I(m)$ is contained in $\Fil^{\max\{0,i-|I|\}}(V)$.

In the same manner as in \cite[p.~35]{Fal89}, we will show that $\alpha_{\Phi,\Psi}$ is parallel with respect to the connection,
satisfies the cocycle condition and induces an equivalence of categories
\begin{equation}
\xymatrix@R=0mm{
\mathcal {MF}_{[a,b]}^{\nabla,\Psi}((\mathcal Y,\mathcal Z)/W)\ar[r] & \mathcal {MF}_{[a,b]}^{\nabla,\Phi}((\mathcal Y,\mathcal Z)/W).\\
(V,\nabla,\Fil,\varphi)\ar@{|->}[r] & (V,\nabla,\Fil,\varphi\circ\alpha_{\Phi,\Psi})\\
}
\end{equation}

Here we verify the necessary formal properties:
\begin{itemize}
\item
For different Frobenius lifts $\Phi_1$, $\Phi_2$, $\Phi_3$, the $\alpha_{\Phi_i,\Phi_j}$ satisfy the transitivity relation.
In other words,
\[
\alpha_{\Phi_1,\Phi_3} = \alpha_{\Phi_2,\Phi_3} \circ \alpha_{\Phi_1,\Phi_2}.
\]
By \eqref{equ_connection}, it suffices to verify that
\[
\frac{z_{\Phi_1,\Phi_3}^K}{K!}=\sum_{I,J} a_{IJ}^K \frac{z_{\Phi_1,\Phi_2}^I}{I!} \frac{z_{\Phi_2,\Phi_3}^J}{J!},
\]
which follows from \autoref{lem_coefficient}.

\item
$\alpha_{\Phi,\Psi}$ is independent of the choice of local coordinates $T_1,\cdots,T_d$:
There is a coordinate-free expression for $\alpha$: Let $J\subset R\otimes_W R$ denote the kernel of multiplication $R\otimes_WR\rightarrow R$, and let $R_1\coloneqq  D_{J}(R\otimes R)^{\wedge}$ be the completed logarithmic divided power hull (see \cite[Definition 5.4]{Kato88}) with respect to the logarithmic poles.
Since $\Phi\cong \Psi\pmod{p}$ and both of them respect the logarithmic poles, the pair $(\Phi,\Psi)$ defines a homomorphism from $R_1$ into $\widehat{R}$.
\begin{equation*}
\xymatrix{
W \ar[r]\ar[d] & \widehat{R} \ar[d] \ar@/^24pt/[dddrr]^{\Psi} &&\\
\widehat{R} \ar[r] \ar@/_24pt/[ddrrr]^\Phi & \widehat{R}\otimes_W\widehat{R} \ar[dr] &&\\
&& R_1\ar[dr]|{(\Phi,\Psi)}&\\
&&& \widehat{R}\\
}
\end{equation*}
The two modules $V\otimes_W \widehat{R}$ and $\widehat{R}\otimes_W V$ become isomorphic over $R_1$
\begin{equation}
\label{stratification}
(V\otimes_W \widehat{R})\otimes_{\widehat{R}\otimes_W\widehat{R}} R_1 \cong (\widehat{R}\otimes_W V)\otimes_{\widehat{R}\otimes_W\widehat{R}} R_1 =: V_1
\end{equation}
by definition of a logarithmic integrable connection.
Denote this module by $V_1$. Moreover, this isomorphism is given by the Taylor formula.
The isomorphism respects the filtration on $V_1$, which is the product of the filtrations given by $\Fil^i(V)$ on $V$ and the divided powers $J^{[n]}$ on $R_1$.

For each $n\in\{0,1,\cdots,p-1\}$, let $\Fil^n(V_1)$ denote either of the following two (canonically isomorphic) submodules,
with isomorphism furnished by \eqref{stratification}
\[\sum_{i+j=n}(\Fil^i V\otimes_W \widehat{R})\otimes_{\widehat{R}\otimes_W\widehat{R}} J^{[j]} \cong \sum_{i+j=n}(\widehat{R}\otimes_W \Fil^i V)\otimes_{\widehat{R}\otimes_W\widehat{R}} J^{[j]}.\]
There exists a map
\[\Fil^n(V_1) \longrightarrow \widetilde{V}\otimes_\Psi \widehat{R}\]
sending $(x_i\otimes_W1) \otimes_{\widehat{R}\otimes\widehat{R}} y_j$ ($x_i\in\Fil^i(V),y_j\in J^{[j]}$) to $x_i\otimes_{\Psi} {(\Phi,\Psi)}(y_j)/p^j$.
That this is well defined follows from a computation in local coordinates.
Now we obtain $\alpha_{\Phi,\Psi}$ by first applying the isomorphism in \eqref{stratification}, and the map above.

\item
That $\alpha_{\Phi,\Psi}$ is parallel for the connections is a straightforward computation.
\end{itemize}
It follows that $\widetilde{V}\otimes_\Phi\widehat{R}$ together with its connection, is up to canonical isomorphism independent of the specific choice of $\Phi$.
\end{proof}

\begin{remark}
\begin{enumerate}
    \item The gluing equivalences in this theorem rely solely on condition~(1) for the Frobenius lifts. Condition~(2) from Setup~\ref{setup:Aff_log_FF} is not required here; it will be used later for the construction of the $\mathbb D^{\log}$‑functor.
    \item When $s=0$, by the formula
\[\nabla\left(\frac{\partial}{\partial T_i}\right)^n=T_i^{-n}\prod_{k=0}^{n-1}\left(\nabla\left(T_i\frac{\partial}{\partial T_i}\right)-k\right),\]
we can verify that \eqref{LogConn} coincides with the formula in \cite[Theorem 2.3]{Fal89}.
\end{enumerate}
\end{remark}

\subsubsection{Global logarithmic Fontaine--Faltings modules}
\label{sect:log_FF_global}

We now drop the assumption that $(\mathcal Y,\mathcal Z)$ is logarithmically small and pass to the global setting. We use the following notation.

\begin{setup}
\label{setup:log_FF}
Let $Y$ be a smooth separated $W$-scheme (not necessarily proper), whose generic and special fibers are both nonempty and geometrically connected, and let $Z\subset Y$ be a relative simple normal crossing divisor. Set $U=Y-Z$. We define $Z_K$, $\mathcal Z$, $\mathcal Z_K$, $U_K$, $\mathcal U$, and $\mathcal U_K$ analogously to the corresponding objects associated with $Y$ in \autoref{setup:scheme_rigid1}. We further set
\[
\mathcal Y_K^\circ\coloneqq\mathcal Y_K-\mathcal Z_K,
\qquad
Y_K^\circ\coloneqq Y_K-Z_K.
\]
\end{setup}

Let $\{\mathcal Y_i\}_{i\in \Lambda}$ be an affine open covering of $\mathcal Y$ such that each pair $(\mathcal Y_i,\mathcal Z\cap\mathcal Y_i)$ is logarithmically small, and choose a Frobenius lift $\Phi_i$ on each $\mathcal Y_i$ satisfying the conditions of Setup~\ref{setup:Aff_log_FF}.
A torsion logarithmic Fontaine--Faltings module over $(\mathcal Y,\mathcal Z)$ with Hodge--Tate weights in $[a,b]$ consists of a filtered logarithmic de Rham module $(V,\nabla,\Fil)$ on $(\mathcal Y,\mathcal Z)$ together with isomorphisms
$\varphi_i:\widetilde V|_{\mathcal Y_i}\otimes_{\Phi_i}\mathcal O_{\mathcal Y_i}\xrightarrow{\sim}V|_{\mathcal Y_i}$
such that
$$(V|_{\mathcal Y_i},\nabla,\Fil,\varphi_i)\in
\mathcal {MF}_{[a,b]}^{\nabla,\Phi_i}
((\mathcal Y_i,\mathcal Z\cap\mathcal Y_i)/W)$$
for every $i$, and such that these local objects correspond to each other on the overlaps under the equivalences of Theorem~\ref{thm_log_version_gluing}.

We denote the resulting category by $\mathcal {MF}_{[a,b]}^{\nabla}
((\mathcal Y,\mathcal Z)/W)$. By Theorem~\ref{thm_log_version_gluing}, it is independent, up to canonical equivalence, of the chosen covering and Frobenius lifts.
Similarly, requiring the local objects to lie in
$\mathcal {MF}_{[a,b]}^{\nabla,\mathrm{tf},\Phi_i}
((\mathcal Y_i,\mathcal Z\cap\mathcal Y_i)/W)$
defines the category of torsion-free logarithmic Fontaine--Faltings modules, denoted by $\mathcal {MF}_{[a,b]}^{\nabla,\mathrm{tf}}
((\mathcal Y,\mathcal Z)/W)$.

\subsection{The logarithmic $\mathbb D^{\log}$-functor}
\label{section logFDF}

We will define the functor $\mathbb D^{\log}$.
This will proceed in two steps. We will first assume that $Y$ is small and $(Y,Z)$ is logarithmically small. Then we will describe the gluing procedure.

\subsubsection{The logarithmic period ring}
\label{section logperiod}

We keep the notation in \autoref{setup:Aff_log_FF}.
Before defining the local logarithmic $\mathbb D^{\log}$-functor, we first
introduce a logarithmic analogue of the period ring $B^+(\widehat R)$
recalled in \autoref{section Period-ring}.

Let $\operatorname{Frac}(\widehat R)$ be the fraction field of $\widehat R$, and fix an algebraic closure
$\overline{\operatorname{Frac}(\widehat R)}$.
Let $\overline{\widehat R}_{\log}$ be the integral closure of $\widehat R$ in the maximal extension of
$\operatorname{Frac}(\widehat R)$ contained in
$\overline{\operatorname{Frac}(\widehat R)}$ for which the normalization of
$\widehat R[1/(pT_1\cdots T_s)]$ in this subextension is unramified over
$\widehat R[1/(pT_1\cdots T_s)]$. Notice that
$\overline{\widehat R}_{\log}
\left[1/(pT_1\cdots T_s)\right]$
is the normalization of
$\widehat R[1/(pT_1\cdots T_s)]$ in the same maximal extension.
Therefore,
\[
\operatorname{Spec}\left(
\overline{\widehat R}_{\log}
[1/(pT_1\cdots T_s)]
\right)
\longrightarrow
\operatorname{Spec}\widehat R
[1/(pT_1\cdots T_s)]
\]
is a profinite \'etale Galois covering.

\begin{lemma}
\label{lem:log_galois_group}
There are natural isomorphisms
\[
\operatorname{Aut}_{\widehat R}
\left(\overline{\widehat R}_{\log}\right)
\cong
\operatorname{Gal}\left(
\overline{\widehat R}_{\log}[1/(pT_1\cdots T_s)]
/
\widehat R[1/(pT_1\cdots T_s)]
\right)
\cong
\pi_1^{\et}(\mathcal Y_K^\circ).
\]
\end{lemma}

\begin{proof}
Since $\overline{\widehat{R}}_{\log}$ is the integral closure of $\widehat{R}$ in the relevant maximal extension of $\operatorname{Frac}(\widehat{R})$, every $\widehat{R}$-automorphism of $\overline{\widehat{R}}_{\log}$ extends uniquely, after localization, to an $\widehat{R}[1/(pT_1\cdots T_s)]$-automorphism of $\overline{\widehat{R}}_{\log}[1/(pT_1\cdots T_s)]$. Conversely, any $\widehat{R}[1/(pT_1\cdots T_s)]$-automorphism of $\overline{\widehat{R}}_{\log}[1/(pT_1\cdots T_s)]$ preserves $\overline{\widehat{R}}_{\log}$, as it respects elements integral over $\widehat{R}$. Therefore
\[
\operatorname{Aut}_{\widehat{R}}\bigl(\overline{\widehat{R}}_{\log}\bigr)
\cong
\operatorname{Gal}\bigl(
\overline{\widehat{R}}_{\log}[1/(pT_1\cdots T_s)]
\,/\,
\widehat{R}[1/(pT_1\cdots T_s)]
\bigr).
\]

Because $\widehat{R}[1/(pT_1\cdots T_s)]$ is a normal integral domain, the standard description of the \'etale fundamental group yields
\[
\operatorname{Gal}\bigl(
\overline{\widehat{R}}_{\log}[1/(pT_1\cdots T_s)]
\,/\,
\widehat{R}[1/(pT_1\cdots T_s)]
\bigr)
\cong
\pi_1^{\et}\bigl(\operatorname{Spec}\,\widehat{R}[1/(pT_1\cdots T_s)]\bigr).
\]
Moreover, $\mathcal Y_K^\circ$ is the rigid analytic open subspace of $\mathcal Y_K$ associated to $\operatorname{Spec}\,\widehat{R}[1/(pT_1\cdots T_s)]$. By \cite[Corollary 2.16]{Han20}, we thus have
\[
\pi_1^{\et}\bigl(\operatorname{Spec}\,\widehat{R}[1/(pT_1\cdots T_s)]\bigr)
\cong
\pi_1^{\et}(\mathcal Y_K^\circ).
\]
The assertion of the lemma now follows.
\end{proof}

We now construct the logarithmic period ring in the same manner as
$B^+(\widehat R)$ in \autoref{section Period-ring}.
Let $\widehat{\overline{\widehat{R}}}_{\log}$ denote the $p$-adic
completion of $\overline{\widehat{R}}_{\log}$. Set
\[
S_{\log}\coloneqq\left(\widehat{\overline{\widehat R}}_{\log}\right)^\flat
\coloneqq \varprojlim_{x\mapsto x^p}\left(
\overline{\widehat R}_{\log}/p\overline{\widehat R}_{\log}
\right).
\]
Choose a compatible system of $p$-power roots
$p,p^{1/p},p^{1/p^2},\cdots$, and let $p^\flat\in S_{\log}$ be the element
corresponding to this system. Let
\[
\theta\colon W(S_{\log}) \longrightarrow
\widehat{\overline{\widehat{R}}}_{\log}
\]
be the Fontaine map.

\begin{lemma}\label{lem:Perfectoid ring}
    $\widehat{\overline{\widehat{R}}}_{\log}$ is a perfectoid ring in the sense of \cite[Definition~3.5]{BMS18} and $I_{\theta}\coloneqq \ker \theta$ is generated by $\xi\coloneqq[p^\flat]-p$, where $[p^\flat]$ denotes the Teichm\"uller lift of $p^\flat$.
\end{lemma}

\begin{proof}
    Take $\pi=p^{1/p}$. By \cite[Lemma~3.10]{BMS18}, it suffices to show that
    \[{\overline{\widehat{R}}}_{\log}/\pi{\overline{\widehat{R}}}_{\log} \longrightarrow
    {\overline{\widehat{R}}}_{\log}/p {\overline{\widehat{R}}}_{\log},
    \quad x\mapsto x^p  \]
    is an isomorphism.

    For injectivity, suppose that
    $x\in\overline{\widehat R}_{\log}$ satisfies
    $x^p\in p\overline{\widehat R}_{\log}$.
    Then $\left(x/\pi\right)^p \in\overline{\widehat R}_{\log}$.
    Since $\overline{\widehat R}_{\log}$ is normal and
    $x/\pi$ lies in its fraction field, the element $x/\pi$
    lies in $\overline{\widehat R}_{\log}$. Hence
    $x\in\pi\overline{\widehat R}_{\log}$.

    For surjectivity, it suffices to show the Frobenius on
    $\overline{\widehat R}_{\log}/p\overline{\widehat R}_{\log}$
    is surjective.
    Let $x\in\overline{\widehat R}_{\log}$. Choose a finite normal
    $\widehat R$-subalgebra
    ${R}'\subset \overline{\widehat R}_{\log}$ containing $x$ and occurring
    in the filtered union defining $\overline{\widehat R}_{\log}$.
    Choose a root $y$ of
    \[
    X^{p^2}+pX-x=0
    \]
    in the fixed algebraic closure of $\operatorname{Frac}(\widehat R)$.
    We claim that $y\in\overline{\widehat R}_{\log}$. Thus the image of
    $y^{p^2}$ in
    $\overline{\widehat R}_{\log}/p\overline{\widehat R}_{\log}$
    is equal to that of $x$. The surjectivity follows.

    We now prove this claim. Consider
    \[
    A'=R'[X]/(X^{p^2}+pX-x).
    \]
    This is finite free over $R'$. The derivative of
    $X^{p^2}+pX-x$ is $p(1+pX^{p^2-1})$, which is a unit after
    inverting $p$. Thus $A'[1/p]$ is \'etale over $R'[1/p]$ and
    $A'[1/(pT_1\cdots T_s)]$ is \'etale over
    $\widehat R[1/(pT_1\cdots T_s)]$.
    Choose a connected component $B'$ of
    $A'[1/(pT_1\cdots T_s)]$ corresponding to the chosen root $y$.
    Let $C'$ be the integral closure of $\widehat R$ in this fraction
    field. Since $B'$ is finite \'etale over
    $\widehat R[1/(pT_1\cdots T_s)]$, it is normal and hence is the
    normalization of $\widehat R[1/(pT_1\cdots T_s)]$ in its fraction
    field. By the definition of $\overline{\widehat R}_{\log}$, $C'$
    embeds into $\overline{\widehat R}_{\log}$. Thus
    $y\in\overline{\widehat R}_{\log}$.

    By the proof of \cite[Lemma~3.10]{BMS18}, we see that
    $\ker\theta$ is generated by $\xi$.
\end{proof}

We define
\[
B^+_{\log}(\widehat{R})
\coloneqq
D_{I_{\theta}}\bigl(W(S_{\log})\bigr)^\wedge
\]
as the completion of the divided power hull
$D_{I_{\theta}}\bigl(W(S_{\log})\bigr)$ with respect to the topology
induced by the divided power ideals
$\bigl({I_{\theta}}+pW(S_{\log})\bigr)^{[n]}$ for $n\ge 1$.

As in the non-logarithmic setting, the Frobenius and the continuous $\pi_1^{\et}(\mathcal Y_K^\circ)$-action extend naturally to
$B^+_{\log}(\widehat R)$ and commute with one another. We continue to
denote the Frobenius on $B^+_{\log}(\widehat R)$ by $\Phi_B$. The
filtrations on $D_{I_\theta}(W(S_{\log}))$ and
$B^+_{\log}(\widehat R)$ are constructed in the same way as their
counterparts in the non-logarithmic case from
\autoref{section Period-ring}. More precisely,
\[
\Fil^n D_{I_{\theta}}(W(S_{\log})) \coloneqq
I_{\theta}^{[n]}\cap
\ker\left(
D_{I_{\theta}}(W(S_{\log}))
\xrightarrow{\mathrm{Frob}}
D_{I_{\theta}}(W(S_{\log}))
\longrightarrow
D_{I_{\theta}}(W(S_{\log}))/p^n
\right),
\]
and $\Fil^n B^+_{\log}(\widehat{R})$ is defined to be the closure of
$\Fil^n D_{I_{\theta}}(W(S_{\log}))$.
Set
\[
\widetilde{B^+_{\log}(\widehat R)}
\coloneqq
\bigoplus_{i\in\mathbb Z}
\Fil^i B^+_{\log}(\widehat R)\big/\sim,
\]
with $px\sim y$ for $x\in\Fil^i B^+_{\log}(\widehat R)$ with $y$ being the image of $b$ under the natural inclusion $\Fil^i B^+_{\log}(\widehat R)\hookrightarrow\Fil^{i-1}B^+_{\log}(\widehat R)$.
For every $i\in\mathbb Z$, define
\[
\varphi_B^i\colon
\Fil^i B^+_{\log}(\widehat R)
\longrightarrow
B^+_{\log}(\widehat R),
\qquad
b\longmapsto
\frac{\Phi_B(b)}{p^i}.
\]
Since the maps $\{\varphi_B^i\}$ are compatible with the transition maps of the filtration, namely $\varphi_B^{i-1}|_{\Fil^i B^+_{\log}(\widehat R)}
=
p\varphi_B^i$,
they induce a $B^+_{\log}(\widehat R)$-linear map
\[
\varphi_B\colon
\widetilde{B^+_{\log}(\widehat R)}
\otimes_{B^+_{\log}(\widehat R),\Phi_B}
B^+_{\log}(\widehat R)
\longrightarrow
B^+_{\log}(\widehat R).
\]

\begin{remark}
\label{rem:compare_log_ordinary_closure}
Set
\[
R^\circ\coloneqq R[1/(T_1\cdots T_s)],
\]
so $R^\circ$ is a small $W$-algebra, and let $\widehat{R}^\circ$ denote
its $p$-adic completion. The natural inclusion
$\widehat{R}\hookrightarrow \widehat{R}^\circ$ induces
$\operatorname{Frac}(\widehat{R})
\hookrightarrow \operatorname{Frac}(\widehat{R}^\circ)$.
Choose compatible algebraic closures, i.e.\ an embedding
$\overline{\operatorname{Frac}(\widehat{R})}
\hookrightarrow
\overline{\operatorname{Frac}(\widehat{R}^\circ)}$
extending this field inclusion. This yields an injection
\[
\overline{\widehat{R}}_{\log}
\hookrightarrow
\overline{\widehat{R}^\circ}.
\]
Indeed, let $L/\operatorname{Frac}(\widehat{R})$ be any finite extension
entering the definition of $\overline{\widehat{R}}_{\log}$, and let $C$
denote the normalization of $\widehat{R}$ in $L$. Then
$C[1/(pT_1\cdots T_s)]$ is finite \'etale over
$\widehat{R}[1/(pT_1\cdots T_s)]$. Since $T_1,\dots,T_s$ are invertible
in $\widehat{R}^\circ$, base change produces a finite \'etale
$\widehat{R}^\circ[1/p]$-algebra. Every element of $C$ is integral over
$\widehat{R}^\circ$, so the image of $C$ lies inside
$\overline{\widehat{R}^\circ}$. As this holds for all such $C$, we get
the desired containment. For $s=0$, the above injection is an
isomorphism.
\end{remark}

\subsubsection{The local $\mathbb D^{\log}$-functor}
\label{section:local_Dlog}

We keep the notation in \autoref{setup:Aff_log_FF}. Set
\[
D_{\log}(\widehat R)
\coloneqq
B^+_{\log}(\widehat R)[1/p]/B^+_{\log}(\widehat R),
\]
equipped with the induced filtration, Frobenius, and
$\pi_1^{\et}(\mathcal Y_K^\circ)$-action.
We first construct a lifting of $\widehat R$ to
$B^+_{\log}(\widehat R)$. Since $$B^+_{\log}(\widehat R)/\Fil^1 B^+_{\log}(\widehat R)\cong \widehat{\overline{\widehat{R}}}_{\log},$$ the inclusion
$\widehat{R}\hookrightarrow \widehat{\overline{\widehat{R}}}_{\log}$ lifts to a homomorphism $\widehat{R}\to B^+_{\log}(\widehat R)$ by the formal smoothness of $\widehat{R}/W$. Such a lifting is not unique.

\begin{lemma}
\label{lem:kappa_Phi_log}
Let $\Phi$ be a Frobenius lifting satisfying the conditions in \autoref{setup:Aff_log_FF}.
\begin{enumerate}
    \item There exists a $W$-algebra homomorphism $\kappa_{\Phi,\log}\colon
\widehat{R}\longrightarrow B^+_{\log}(\widehat{R})$
lifting the canonical inclusion $\widehat{R}\longrightarrow
\widehat{\overline{\widehat{R}}}_{\log}$
and satisfying $\Phi_B\circ\kappa_{\Phi,\log}
=
\kappa_{\Phi,\log}\circ\Phi$.
    \item Let $\sigma \in \pi_1^{\et}(\mathcal Y_K^\circ)$, then for all $1\leq k \leq s$,
    \[\sigma\circ \kappa_{\Phi,\log}(T_k)/\kappa_{\Phi,\log}(T_k)\in 1+I_{\theta}.\]

    \item Let $\Psi$ be another Frobenius lifting satisfying the conditions in \autoref{setup:Aff_log_FF}. Then for all $1\le i\le s$,
\[
\kappa_{\Phi,\log}(T_i)/\kappa_{\Psi,\log}(T_i)\in 1+{I_{\theta}}.
\]
\end{enumerate}
\end{lemma}
\begin{proof}
For (1), the construction of $\kappa_{\Phi,\log}$ is the logarithmic analogue of \cite[p.~36]{Fal89}. Let
\[
R_{\infty}\coloneqq \varinjlim \left(\widehat{R}
\xrightarrow{\Phi}
\widehat{R}
\xrightarrow{\Phi}
\widehat{R}
\xrightarrow{\Phi}\cdots\right).
\]
By \autoref{setup:Aff_log_FF}(2), the transition maps become finite \'etale after inverting $pT_1\cdots T_s$. Hence $R_{\infty}$ embeds into $\overline{\widehat{R}}_{\log}$.

Since the Frobenius on $R_{\infty}/pR_{\infty}$ is bijective, its $p$-adic completion $\widehat{R}_{\infty}$ is isomorphic to $W(R_{\infty}/pR_{\infty})$, compatibly with the canonical Frobenius. The inclusion $R_{\infty}\hookrightarrow \overline{\widehat{R}}_{\log}$ induces a homomorphism
\[
R_{\infty}\twoheadrightarrow R_{\infty}/pR_{\infty}\cong \varprojlim_{x\mapsto x^p}R_{\infty}/pR_{\infty}
\longrightarrow \varprojlim_{x\mapsto x^p}\overline{\widehat{R}}_{\log}/p\overline{\widehat{R}}_{\log}=S_{\log}.
\]
Composing the homomorphisms
\[
\widehat{R}\hookrightarrow \widehat{R}_{\infty}\cong W(R_{\infty}/pR_{\infty})\longrightarrow W(S_{\log}) \longrightarrow B^+_{\log}(\widehat{R})
\]
gives the desired homomorphism $\kappa_{\Phi,\log}$.

For (2), let $1\leq k\leq s$. Write $\Phi(T_k)=w_kT_k^p$ and $\Psi(T_k)=w_k'T_k^p$, where $w_k,w_k'\in\widehat R^\times$. Since both $\Phi$ and $\Psi$ lift the absolute Frobenius, $w_k\equiv w_k'\equiv1\pmod p$.

We first describe more explicitly the element of $S_{\log}$ arising in the construction of $\kappa_{\Phi,\log}(T_k)$. Since $\Phi$ is an automorphism of $R_{\infty}$, applying $\Phi^{-n}$ to $\Phi(T_k)=w_kT_k^p$ gives
\[
\Phi^{-(n-1)}(T_k)=\Phi^{-n}(w_k)\bigl(\Phi^{-n}(T_k)\bigr)^p.
\]
Iterating, we obtain
\begin{equation}\label{equ_Frobenius}T_k= \left( \prod_{m=1}^{n} \Phi^{-m}(w_k)^{p^{m-1}} \right) \bigl(\Phi^{-n}(T_k)\bigr)^{p^n}.
\end{equation}

Let $\overline{T}_k$ be the image of $T_k$ in $R_{\infty}/pR_{\infty}$. By \cite[Proposition~1.42]{FO22}, under the isomorphism $\widehat R_{\infty}\cong W(R_{\infty}/pR_{\infty})$, the limit $\lim_{n\to\infty}\bigl(\Phi^{-n}(T_k)\bigr)^{p^n}$ maps to the Teichm\"uller lift $[\overline{T}_k]$ of $\overline{T}_k$. Its image in $W(S_{\log})$ is $[T_k^\flat]$, where
\[
T_k^\flat\coloneqq (\overline{T}_k,\overline{\Phi^{-1}(T_k)},\overline{\Phi^{-2}(T_k)},\ldots)\in S_{\log},
\]
and $\overline{\Phi^{-n}(T_k)}$ denotes the image of $\Phi^{-n}(T_k)$ in $R_{\infty}/pR_{\infty}$.

Since $\Phi^{-m}(w_k)\equiv1\pmod p$, one has $\Phi^{-m}(w_k)^{p^{m-1}}\equiv1\pmod{p^m}$. Hence the infinite product $\prod_{m\geq1}\Phi^{-m}(w_k)^{p^{m-1}}$ converges $p$-adically to a unit. Let $c_{\Phi,k}$ denote its image in $W(S_{\log})$. Passing to the limit in \eqref{equ_Frobenius}, we obtain
\[\kappa_{\Phi,\log}(T_k)=c_{\Phi,k}[T_k^\flat], \qquad c_{\Phi,k}\in W(S_{\log})^\times.
\]

Since $\sigma(T_k^\flat)/T_k^\flat$ is represented by a compatible sequence of $p$-power roots of unity in $\overline{\widehat{R}}_{\log}/p\overline{\widehat{R}}_{\log}$, it belongs to $S_{\log}^\times$. Thus
\[\sigma\circ\kappa_{\Phi,\log}(T_k)/\kappa_{\Phi,\log}(T_k)=\sigma(c_{\Phi,k})/c_{\Phi,k}\cdot[\sigma(T_k^\flat)/T_k^\flat]\in W(S_{\log})^\times.\]
Moreover, since $\theta$ is $\sigma$-equivariant and $\sigma$ fixes $T_k$, we have $$\theta(\sigma\circ\kappa_{\Phi,\log}(T_k))=T_k=\theta(\kappa_{\Phi,\log}(T_k)).$$ Therefore $\sigma\circ\kappa_{\Phi,\log}(T_k)/\kappa_{\Phi,\log}(T_k)\in1+I_{\theta}$.
This proves (2).

For (3), applying the same argument to $\Psi$, we obtain
\[
\kappa_{\Psi,\log}(T_k)=c_{\Psi,k}[T_k'^\flat], \qquad c_{\Psi,k}\in W(S_{\log})^\times,
\]
where $T_k'^\flat\coloneqq(\overline{T}_k,\overline{\Psi^{-1}(T_k)},\overline{\Psi^{-2}(T_k)},\ldots)\in S_{\log}$.

The quotient $T_k^\flat/T_k'^\flat$ is represented by a compatible sequence of $p$-power roots of unity in $\overline{\widehat{R}}_{\log}/p\overline{\widehat{R}}_{\log}$, and hence belongs to $S_{\log}^\times$. Thus
\[
\kappa_{\Phi,\log}(T_k)/\kappa_{\Psi,\log}(T_k)=c_{\Phi,k}/c_{\Psi,k}\cdot[T_k^\flat/T_k'^\flat]\in W(S_{\log})^\times.\]
Since both $\kappa_{\Phi,\log}$ and $\kappa_{\Psi,\log}$ lift the canonical inclusion $\widehat R\to\widehat{\overline{\widehat R}}_{\log}$, we have $$\theta(\kappa_{\Phi,\log}(T_k))=\theta(\kappa_{\Psi,\log}(T_k))=T_k.$$ Therefore $\kappa_{\Phi,\log}(T_k)/\kappa_{\Psi,\log}(T_k)\in1+I_{\theta}$. This proves (3).
\end{proof}

Let $M=(V,\nabla,\Fil,\varphi) \in \mathcal{MF}_{[0,p-2]}^{\nabla,\Phi} ((\mathcal Y,\mathcal Z)/W)$. We define
\[
\mathbb P^{\cris}_{\Phi,\log}(M)
\coloneqq
V\otimes_{\widehat R,\kappa_{\Phi,\log}}
B^+_{\log}(\widehat R).
\]
It is equipped with the tensor product filtration
\[
\Fil^r\mathbb P^{\cris}_{\Phi,\log}(M)
=
\sum_{i=0}^{p-2}
\Fil^iV\otimes_{\widehat R,\kappa_{\Phi,\log}}
\Fil^{r-i}B^+_{\log}(\widehat R).
\]
Set
\[
\widetilde{\mathbb P^{\cris}_{\Phi,\log}(M)}
\coloneqq
\bigoplus_{r\in\mathbb Z}
\Fil^r\mathbb P^{\cris}_{\Phi,\log}(M)\big/\sim,
\]
with $px\sim y$ for $x\in \Fil^r\mathbb P^{\cris}_{\Phi,\log}(M)$ with $y$ being the image of $x$ under the natural inclusion $\Fil^r\mathbb P^{\cris}_{\Phi,\log}(M)\hookrightarrow\Fil^{r-1}\mathbb P^{\cris}_{\Phi,\log}(M)$.
For every $r\in\mathbb Z$, $v\in\Fil^iV$ and
$b\in\Fil^{r-i}B^+_{\log}(\widehat R)$, define
\[
\varphi_{\mathbb P}^r\colon
\Fil^r\mathbb P^{\cris}_{\Phi,\log}(M)
\longrightarrow
\mathbb P^{\cris}_{\Phi,\log}(M),
\qquad
v\otimes b
\longmapsto
\varphi^i(v)\otimes \varphi_B^{r-i}(b),
\]
Since the maps $\{\varphi_{\mathbb P}^r\}$ are compatible with the
transition maps of the filtration, namely $\varphi_{\mathbb P}^{r-1}|_{\Fil^r\mathbb P^{\cris}_{\Phi,\log}(M)} =p\varphi_{\mathbb P}^r$,
they induce a $B^+_{\log}(\widehat R)$-linear map
\[
\varphi_{\mathbb P}\colon
\widetilde{\mathbb P^{\cris}_{\Phi,\log}(M)}
\otimes_{B^+_{\log}(\widehat R),\Phi_B}
B^+_{\log}(\widehat R)
\longrightarrow
\mathbb P^{\cris}_{\Phi,\log}(M).
\]

We next define the action of
$\pi_1^{\et}(\mathcal Y_K^\circ)$. We use the notation  $\nabla(\partial)^{\{I\}}$ from the proof of \autoref{thm_log_version_gluing}.
 For $\sigma\in\pi_1^{\et}(\mathcal Y_K^\circ)$, $v\in V$ and $b\in B^+_{\log}(\widehat{R})$, the action of $\pi_1^{\et}(\mathcal Y_K^\circ)$ is given by the logarithmic Taylor expansion
\begin{equation}
\label{equ:log_Galois_Taylor}
\sigma(v\otimes b)
=
\left(
\sum_{I\in\mathbb N^d}
\nabla(\partial)^{\{I\}}(v)\otimes
\frac{\Delta_{\Phi,\log}(\sigma)^I}{I!}
\right)\sigma(b),
\end{equation}
where for a multi‑index $I=(i_1,\dots,i_d)\in\mathbb N^d$, $\Delta_{\Phi,\log}(\sigma)^I \coloneqq \prod_{k=1}^d \big(\Delta_{\Phi,\log,k}(\sigma)\big)^{i_k}$ with
\[
\Delta_{\Phi,\log,k}(\sigma) =
\begin{cases}
\sigma\circ \kappa_{\Phi,\log}(T_k)/\kappa_{\Phi,\log}(T_k) - 1, & 1\le k\le s,\\
\sigma\circ \kappa_{\Phi,\log}(T_k)-\kappa_{\Phi,\log}(T_k), & s+1\le k\le d.
\end{cases}
\]
\eqref{equ:log_Galois_Taylor} is well-defined. Indeed, by \autoref{lem:kappa_Phi_log}(2), we have $\Delta_{\Phi,\log,k}(\sigma)\in I_{\theta}$ for $1\leq k\leq d$. Thus $\Delta_{\Phi,\log}(\sigma)^I/I!\in I_{\theta}^{[|I|]}$. Since $B^+_{\log}(\widehat R)$ is complete with respect to the topology defined by the ideals $I_{\theta}^{[n]}B^+_{\log}(\widehat R)$, the series \eqref{equ:log_Galois_Taylor} converges.
To show \eqref{equ:log_Galois_Taylor} gives an action of $\pi_1^{\et}(\mathcal Y_K^\circ)$, it suffices to verify that for any $\sigma,\tau \in \pi_1^{\et}(\mathcal Y_K^\circ)$, we have
\[
\begin{cases}
\displaystyle
1+\Delta_{\Phi,\log,k}(\sigma\tau)
=
\bigl(1+\Delta_{\Phi,\log,k}(\sigma)\bigr)
\bigl(1+\sigma(\Delta_{\Phi,\log,k}(\tau))\bigr),
& 1\leq k\leq s,\\[10pt]
\displaystyle
\Delta_{\Phi,\log,k}(\sigma\tau)
=
\Delta_{\Phi,\log,k}(\sigma)
+\sigma(\Delta_{\Phi,\log,k}(\tau)),
& s+1\leq k\leq d,
\end{cases}
\]
which holds by the definitions.
The action is continuous. Indeed, the natural action of
$\pi_1^{\et}(\mathcal Y_K^\circ)$ on $B^+_{\log}(\widehat R)$ is continuous, so $\sigma(b)\to b$ and $\Delta_{\Phi,\log,k}(\sigma)\to0$ as $\sigma\to1$.
Since $\Delta_{\Phi,\log}(\sigma)^I/I!\in I_\theta^{[|I|]}$, modulo $I_\theta^{[n]}B^+_{\log}(\widehat R)$ the Taylor expansion \eqref{equ:log_Galois_Taylor} contains only finitely many terms.
For $\sigma$ sufficiently close to $1$, all the remaining terms with $I\neq0$ also lie in $I_\theta^{[n]}B^+_{\log}(\widehat R)$.
Hence $\sigma(v\otimes b)\to v\otimes b$ for every $v\in V$ and $b\in B^+_{\log}(\widehat R)$ as $\sigma\to 1$, and therefore the action on $\mathbb P^{\cris}_{\Phi,\log}(M)$ is continuous.

We can now define the local logarithmic $\mathbb D$-functor. Set
\[
\mathbb D^{\log}_{\Phi}(M)
\coloneqq
\mathrm{Hom}
\left(
\mathbb P^{\cris}_{\Phi,\log}(M),
D_{\log}(\widehat R)
\right),
\]
where the homomorphisms are $B^+_{\log}(\widehat R)$-linear and compatible with the filtrations and Frobenius structures. The fundamental group acts by $(\sigma f)(x)
\coloneqq
\sigma\bigl(f(\sigma^{-1}(x))\bigr).$
Note that if we endow $D_{\log}(\widehat{R})$ with the $R$-module structure through $\kappa_{\Phi,\log}$, then \[\mathbb D^{\log}_{\Phi}(M)=\operatorname{Hom}_{\MFbig{R}}(M,D_{\log}(\widehat{R})).\]

The torsion-free version is obtained by passing to the inverse limit. Accordingly, for any $M\in
\mathcal{MF}_{[0,p-2]}^{\nabla,\mathrm{tf},\Phi}
((\mathcal Y,\mathcal Z)/W)$, we define
\[
\mathbb D^{\log}_{\Phi}(M)
\coloneqq
\varprojlim_n
\mathbb D^{\log}_{\Phi}(M/p^nM).
\]
Thus, in both the torsion and torsion-free cases, we obtain contravariant functors from logarithmic Fontaine--Faltings modules with Hodge--Tate weights in $[0,p-2]$ to
$\mathbb Z_p$-representations of
$\pi_1^{\et}(\mathcal Y_K^\circ)$.

We next establish the logarithmic analogue of the reduction used by Faltings in the proof of \cite[Theorem~2.4]{Fal89}, which allows us to replace the period ring modulo $p$ by a simpler object.
\begin{lemma}
\label{lem:Bpluslog_mod_p}
  Set $E_{\log} \coloneqq
\overline{\widehat R}_{\log}/
p\overline{\widehat R}_{\log}$.
We endow $E_{\log}$ with the $R/pR$-module structure induced by
\[
R/pR
\longrightarrow
R_\infty/pR_\infty
\xrightarrow{\Phi^{-1}}
R_\infty/pR_\infty
\longrightarrow
E_{\log}.
\]
Choose a $p$-th root $p^{1/p}\in\overline{\widehat R}_{\log}$.
Define
\[
F^iE_{\log} =
\begin{cases}
E_{\log}, & i\leq  0,\\[4pt]
p^{i/p}\overline{\widehat R}_{\log}/
p\overline{\widehat R}_{\log},
& 1\leq i\leq p-2,\\[4pt]
0, & i\geq p-1.
\end{cases}
\]
For $ i\leq p-2$, define
\[
\varphi^i_E: F^iE_{\log}\longrightarrow E_{\log},
\quad x\mapsto \frac{x^p}{(-p)^i}.
\]
These structures make $E_{\log}$ an object of $\MFbig{R}$.
Moreover, there exists a surjective map
$ B^+_{\log}(\widehat R)/ pB^+_{\log}(\widehat R) \longrightarrow E_{\log}$ inducing an isomorphism
\[
\mathbb D^{\log}_{\Phi}(M)
\xrightarrow{\sim}
\operatorname{Hom}_{\MFbig{R}}(M,E_{\log})\]
for any $p$-torsion $M
\in
\mathcal{MF}_{[0,p-2]}^{\nabla,\Phi}
((\mathcal Y,\mathcal Z)/W)$.
\end{lemma}

\begin{proof}
For simplicity, write
\[
B/pB\coloneqq
B^+_{\log}(\widehat R)/pB^+_{\log}(\widehat R).
\]
Since $D_{\log}(\widehat R)[p]\cong B/pB$ and $pM=0$, every morphism
$M\to D_{\log}(\widehat R)$ factors through its $p$-torsion. Hence
\begin{equation}\label{equ:D_B/pB}
 \mathbb D^{\log}_{\Phi}(M) \cong \operatorname{Hom}_{\MFbig{R}}(M,B/pB) \cong \operatorname{Hom}_{\MFbig{R}} \big(M,(B/pB)_{\leq p-2}\big),
\end{equation}
where the second isomorphism follows from $F^iM=0$ for
$i\geq p-1$.

We next compare the truncated object $(B/pB)_{\leq p-2}$ with $E_{\log}$. The projection
\[
S_{\log}=\bigl(\overline{\widehat R}_{\log}\bigr)^\flat
\longrightarrow E_{\log},
\qquad
(x_0,x_1,\ldots)\longmapsto x_1,
\]
induces a homomorphism $W(S_{\log})\longrightarrow E_{\log}$ which maps
$\xi$ to $p^{1/p}$. It extends to a continuous surjective ring homomorphism
$\pi:B/pB\longrightarrow E_{\log}$ characterized by
\[
\pi\bigl(\gamma_n(\xi)\bigr)=
\begin{cases}
\dfrac{p^{n/p}}{n!},&0\leq n\leq p-1,\\[8pt]
0,&n\geq p,
\end{cases}
\]
where $\gamma_n(\xi)$ denotes the $n$-th divided power of $\xi$.

We claim that $\pi$ induces a morphism $(B/pB)_{\leq p-2}\to E_{\log}$ in $\MFbig{R}$, which we also denote by $\pi$. The compatibility with
the filtration follows directly from the above description of $\pi$.
Indeed, for $0\leq i\leq p-2$, the image of $F^i$ is contained in
$F^iE_{\log}$, while for $i\geq p-1$ both filtration structures are
zero.
Recall that the Frobenius structure of
$B^+_{\log}(\widehat R)$ is given by
\[
\varphi_B^i \colon \Fil^i B^+_{\log}(\widehat R)
\longrightarrow B^+_{\log}(\widehat R),\qquad
x\longmapsto \frac{\Phi_B(x)}{p^i},
\]
and let $\{\varphi_{B/pB}^i\}$ be the induced
Frobenius structure on $(B/pB)_{\leq p-2}$. Then it suffices
to verify that, for $0\leq i\leq p-2$ and $n\geq i$,
\begin{equation}\label{equ:Frob_B/pB_E}
\varphi_E^i\circ \pi\bigl(\gamma_n(\xi)\bigr)
=
\pi\circ\varphi_{B/pB}^i\bigl(\gamma_n(\xi)\bigr).
\end{equation}

Notice that $\Phi_B(\xi)/p$ belongs to
$B^+_{\log}(\widehat R)$ and its image in $E_{\log}$ is $-1$.
For $i\leq n\leq p-1$, a direct computation gives
\[
\varphi_E^i\circ \pi\bigl(\gamma_n(\xi)\bigr)
=
\pi\circ\varphi_{B/pB}^i\bigl(\gamma_n(\xi)\bigr)
=
(-1)^n\frac{p^{n-i}}{n!}
=
\begin{cases}
0,&n>i,\\[4pt]
\dfrac{(-1)^i}{i!},&n=i.
\end{cases}
\]
For $0\leq i\leq p-2$ and $n\geq p$, we have
$\pi(\gamma_n(\xi))=0$, so the left hand side of
\eqref{equ:Frob_B/pB_E} is zero. To show that the right hand side is
also zero, it is enough to show that
$\varphi_B^i(\gamma_n(\xi))\in pB^+_{\log}(\widehat R)$. Since
\[
\varphi_B^i(\gamma_n(\xi))
=
\left(\frac{\Phi_B(\xi)}{p}\right)^n
\frac{p^{n-i}}{n!}
\]
and $n-i-v_p(n!)\geq1$ for $0\leq i\leq p-2$ and $n\geq p$, the claim
follows. Thus $\pi$ is a morphism in $\MFbig{R}$.

Let $X$ be the kernel of $\pi$. Then $X$ is an object in
$\MFbig{R}$ and, by the above description of $\pi$, its underlying
module is the closed ideal generated by $\gamma_n(\xi)$ for $n\geq p$.
Since $X\subseteq F^{p-2}(B/pB)$, we have
\[
F^iX=
\begin{cases}
X,&i\leq p-2,\\
0,&i\geq p-1.
\end{cases}
\]
Let $\{\varphi_X^i\}$ be its Frobenius structure.
The preceding computation shows that $\varphi_X^i=0$ for
$i\leq p-2$, and for $i\geq p-1$ this is automatic since $F^iX=0$.
Hence $\varphi_X^i=0$ for all $i$.

As the images of the divided Frobenius maps of $M$ in degrees
$0,\ldots,p-2$ generate $M$, every morphism $M\to X$ is zero. Thus
$\operatorname{Hom}_{\MFbig{R}}(M,X)=0$.

We also have $\operatorname{Ext}^1_{\MFbig{R}}(M,X)=0$. Indeed,
consider an extension
$$0\to X\to P\to M\to0$$ in $\MFbig{R}$. For $m\in F^iM$, choose a
lift $\widetilde m\in F^iP$. Since $F^iX=X$ and
$\varphi_X^i=0$ for $i\leq p-2$, the element
$\varphi_P^i(\widetilde m)$ is independent of the choice of
$\widetilde m$. These maps are compatible with the transition maps
and therefore induce a map
$\widetilde M\otimes_{R,\Phi}R\to P$. Its composition with
$P\to M$ is the strong-divisibility isomorphism
$\widetilde M\otimes_{R,\Phi}R\to M$, and hence gives a section
$M\to P$. Since $F^iX=X$ for $i\leq p-2$, this section is compatible
with the filtration, and by construction it commutes with the divided
Frobenius maps. Thus the extension splits. Therefore
$\operatorname{Ext}^1_{\MFbig{R}}(M,X)=0.$

Applying $\operatorname{Hom}_{\MFbig{R}}(M,-)$ to
\begin{equation}\label{equ:exactseq}
0\longrightarrow X
\longrightarrow
(B/pB)_{\leq p-2}
\longrightarrow
E_{\log}
\longrightarrow0
\end{equation}
gives
\[
\operatorname{Hom}_{\MFbig{R}}
\bigl(M,(B/pB)_{\leq p-2}\bigr)
\xrightarrow{\sim}
\operatorname{Hom}_{\MFbig{R}}(M,E_{\log}).
\]
Combining this isomorphism with \eqref{equ:D_B/pB} proves the assertion.
\end{proof}
\begin{remark}
In the proof of \cite[Theorem~2.4]{Fal89}, Faltings directly uses the surjection $B^+(\widehat{R})/pB^+(\widehat{R})\to E$
as a morphism in $\MFbig{R}$, without truncating at degree
$p-2$; see \cite[p.~37]{Fal89}. Strictly speaking, this is not correct: although the underlying ring homomorphism is well-defined, it is not compatible with the Frobenius structure in degree $p-1$.
\end{remark}

\begin{lemma}
\label{lem:Dlog_Ext_vanishing}
Let $M$ be a $p^n$-torsion object in
$\mathcal{MF}_{[0,p-2]}^{\nabla,\Phi}((\mathcal Y,\mathcal Z)/W)$ admitting an adapted basis.
Then $\operatorname{Ext}^1_{\MFbig{R}}
\bigl(M,D_{\log}(\widehat R)\bigr)=0$.
\end{lemma}

\begin{proof}
We first reduce to the case $pM=0$. Suppose that $p^nM=0$. From the exact sequence
\[
0\longrightarrow p^{n-1}M
\longrightarrow M
\longrightarrow M/p^{n-1}M
\longrightarrow0
\]
and the associated long exact sequence for
$\operatorname{Hom}_{\MFbig{R}}(-,D_{\log}(\widehat R))$, the vanishing for $M$ follows from the vanishing for
$p^{n-1}M$ and $M/p^{n-1}M$. Induction on $n$ therefore reduces us
to the case $pM=0$.

Assume now that $pM=0$, and consider an extension
\[
0\longrightarrow D_{\log}(\widehat R)
\longrightarrow P
\longrightarrow M
\longrightarrow0
\]
in $\MFbig{R}$. We need to show that it splits.
Since $D_{\log}(\widehat R)$ is $p$-divisible, multiplication by $p$ on
$D_{\log}(\widehat R)$ is surjective.
Hence, by the snake lemma, multiplication by $p$ on the extension induces an exact sequence
\begin{equation}\label{equ: exactsq1}
0\longrightarrow B/pB
\longrightarrow P[p]
\longrightarrow M
\longrightarrow0.
\end{equation}
Since $F^iM=0$ for $i\geq p-1$, it suffices to   split the truncation of
\eqref{equ: exactsq1} at degree $p-2$, namely
\begin{equation}\label{equ: exactsq11}
0\longrightarrow (B/pB)_{\leq p-2}
\longrightarrow P[p]_{\leq p-2}
\longrightarrow M
\longrightarrow0.
\end{equation}
The isomorphism class of the extension \eqref{equ: exactsq11} is an element of $\operatorname{Ext}^1_{\MFbig{R}}
\bigl(M,(B/pB)_{\leq p-2}\bigr)$.
By the discussion in the proof of \autoref{lem:Bpluslog_mod_p}, we have
$\operatorname{Ext}^1_{\MFbig{R}}(M,X)=0$, thus there is an injection
\[
\operatorname{Ext}^1_{\MFbig{R}}
\bigl(M,(B/pB)_{\leq p-2}\bigr)
\hookrightarrow
\operatorname{Ext}^1_{\MFbig{R}}(M,E_{\log}).
\]
Pushing out \eqref{equ: exactsq11} along the surjection
$(B/pB)_{\leq p-2} \longrightarrow E_{\log}$
yields an extension
\begin{equation}\label{equ: exactsq2}
0\longrightarrow E_{\log}
\longrightarrow P_E
\longrightarrow M
\longrightarrow0.
\end{equation}

Since \eqref{equ: exactsq11} is annihilated by $p$, so is its pushout.
Thus the above sequence is an extension of filtered $R/pR$-modules.
It remains to show that \eqref{equ: exactsq2} splits in $\MFbig{R}$.

After forgetting the Frobenius structures,
\eqref{equ: exactsq2} splits as an extension of filtered
$R/pR$-modules. Indeed, the sequence is strict, and each
$\operatorname{gr}^i_F M$ is a finite free $R/pR$-module.
Hence the induced surjections
$\operatorname{gr}^i_F P_E\longrightarrow
\operatorname{gr}^i_F M$ admit sections. Lifting these sections successively along the
filtration gives a filtered splitting $s\colon M\to P_E$.

Let
$\{m_1,\ldots,m_h\}$ be an adapted basis of $M$. Let
$e_i\in[0,p-2]$ be determined by
$m_i\in\Fil^{e_i}M-\Fil^{e_i+1}M$. Write
\[
\varphi^{e_i}(m_i)=\sum_{j=1}^h a_{ij}m_j,
\qquad a_{ij}\in R/pR.
\]
The matrix $(a_{ij})$ is invertible. Since the Frobenius on
$R_\infty/pR_\infty$ is bijective, choose
$b_{ij}\in R_\infty\subset\overline{\widehat R}_{\log}$ such that
$b_{ij}^p\equiv a_{ij}\pmod p$.
Choose $y_i\in\overline{\widehat R}_{\log}$ such that
\[
\sum_j a_{ij}s(m_j)-\varphi^{e_i}(s(m_i))
=y_i\pmod p
\]
in $E_{\log}$. Let $s':M\to P_E$ be another splitting with
\[
s'(m_i)=s(m_i)+x_i\pmod p,
\qquad
x_i\in p^{e_i/p}\overline{\widehat R}_{\log}.
\]
Then $s'$ respects the Frobenius structure if and only if
\[
\varphi^{e_i}(s'(m_i))
=
s'(\varphi^{e_i}(m_i)),
\qquad 1\leq i\leq h,
\]
or equivalently,
\[
x_i^p\equiv
(-p)^{e_i}\left(y_i+\sum_jb_{ij}x_j\right)
\pmod{p^{e_i+1}},
\qquad1\leq i\leq h.
\]

We claim that these congruences always admit a solution.
Indeed, choose a finite normal $\widehat R$-subalgebra
$R''\subset\overline{\widehat R}_{\log}$ occurring in the
definition of $\overline{\widehat R}_{\log}$ and containing all the
$b_{ij}$ and $y_i$. Then
$R''[1/(pT_1\cdots T_s)]$ is finite \'etale over
$\widehat R[1/(pT_1\cdots T_s)]$. Consider the $\widehat R$-algebra
\[
A''\coloneqq
R''[X_1,\ldots,X_h]\Big/
\bigl(
X_i^p-(-p)^{e_i}(y_i+\sum_jb_{ij}X_j)
\bigr)_{1\leq i\leq h}.
\]
Then $A''$ is finite flat over $R''$. We claim that it becomes finite
\'etale after inverting $p$. The corresponding Jacobian matrix is
 \[
(-p)^{e_i}\left(-(b_{ij})+
\operatorname{diag}\left(p(-p)^{-e_i}X_i^{p-1}\right)\right).
\]
Since $(b_{ij})$ is invertible and
\[
\left(p(-p)^{-e_i}X_i^{p-1}\right)^p
=
(-1)^{e_i}p^{p-e_i}
\left(y_i+\sum_jb_{ij}X_j\right)^{p-1}
\in p^2A'',
\]
the diagonal terms are $p$-adically topologically nilpotent. Hence the
Jacobian matrix is invertible after inverting $p$, and therefore
$A''[1/p]$ is finite \'etale over $R''[1/p]$. Consequently,
$A''[1/(pT_1\cdots T_s)]$ is finite \'etale over
$\widehat R[1/(pT_1\cdots T_s)]$.
Choose a connected component $B''$ of
$A''[1/(pT_1\cdots T_s)]$ and an embedding of its fraction field into
the fixed algebraic closure of $\operatorname{Frac}(\widehat R)$.
Let $C''$ be the integral closure of $\widehat R$ in this fraction field.
By the definition of $\overline{\widehat R}_{\log}$, $C''$ embeds into
$\overline{\widehat R}_{\log}$, and the image of $X_i$ gives a solution
of the above congruences. We may therefore modify $s$ so that it is
compatible with the Frobenius structures. Hence
\eqref{equ: exactsq2}, and therefore \eqref{equ: exactsq1}, splits. Thus every class in
$\operatorname{Ext}^1_{\MFbig{R}}
\bigl(M,D_{\log}(\widehat R)\bigr)$
vanishes.
\end{proof}

Using this description of $\mathbb D^{\log}_{\Phi}$ for $p$-torsion objects together with its exactness, we may now compare this logarithmic functor with Faltings' classical $\mathbb D$-functor in the logarithmically small setting.

\begin{proposition}
\label{lem_affDLog}
We keep the notation of \autoref{setup:Aff_log_FF}. Let
$M=(V,\nabla,\Fil,\varphi)$ be an object of either
$\mathcal{MF}_{[0,p-2]}^{\nabla,\Phi}((\mathcal Y,\mathcal Z)/W)$ or $\mathcal{MF}_{[0,p-2]}^{\nabla,\mathrm{tf},\Phi}((\mathcal Y,\mathcal Z)/W)$,
and assume that $M$ admits an adapted basis. Then the following hold.
\begin{enumerate}
    \item There exists a canonical isomorphism
    \[
    \mathbb D^{\log}_{\Phi}(M)
    |_{\mathcal U_K}
    \cong
    \mathbb D(M|_{\mathcal U}).
    \]

   \item Let $Y'$ be an affine open subset of $Y$, and let $\mathcal Y'$ be its $p$-adic completion.
   Denote by $\mathcal Y'_K$ the rigid generic fiber of
$\mathcal Y'$, and set $\mathcal Y_K^{\prime\circ} \coloneqq \mathcal Y'_K\cap \mathcal Y_K^\circ$. Then the Frobenius lifting $\Phi$ restricts uniquely to $\mathcal Y'$, and there exists a canonical isomorphism
    \[
    \mathbb D^{\log}_{\Phi}(M|_{\mathcal Y'})
    \cong
    \mathbb D^{\log}_{\Phi}(M)|_{\mathcal Y_K^{\prime\circ}}.
    \]
\end{enumerate}
\end{proposition}

\begin{proof}
 For (1), it suffices to prove the statement for $p^n$-torsion objects, the torsion-free case then follows by applying the torsion case to $M/p^nM$ and passing to the inverse limit. Recall that $R^\circ=R[1/(T_1\cdots T_s)]$ is a small $W$-algebra. After choosing compatible algebraic closures as in \autoref{rem:compare_log_ordinary_closure}, the inclusion $\overline{\widehat R}_{\log}\longrightarrow\overline{\widehat R^\circ}$ induces a homomorphism
\begin{equation}
\label{equ:Bplus_log_to_Bplus}
B^+_{\log}(\widehat R)
\longrightarrow
B^+(\widehat R^\circ),
\end{equation}
compatible with Frobenius and filtrations, and equivariant for $\pi_1^{\et}(\mathcal U_K)\to\pi_1^{\et}(\mathcal Y_K^\circ)$.
The Frobenius $\Phi$ extends to $R^\circ$, and the constructions of $\kappa_{\Phi,\log}$ and the usual $\kappa_\Phi$ are compatible with \eqref{equ:Bplus_log_to_Bplus}. Hence
\begin{equation}
\label{equ:Plog_to_P}
\mathbb P^{\cris}_{\Phi,\log}(M)
\otimes_{B^+_{\log}(\widehat R)}
B^+(\widehat R^\circ)
\cong
\Pcris{M|_{\mathcal U}},
\end{equation}
compatibly with filtrations, Frobenius structures, and the $\pi_1^{\et}(\mathcal U_K)$-actions. In particular, \eqref{equ:Bplus_log_to_Bplus} and \eqref{equ:Plog_to_P} induce a natural $\pi_1^{\et}(\mathcal U_K)$-equivariant morphism
\begin{equation}
\label{equ:Dlog_to_D}
\left.
\mathbb D_{\Phi}^{\log}(M)
\right|_{\mathcal U_K}
\longrightarrow
\mathbb D(M|_{\mathcal U}).
\end{equation}

To show that \eqref{equ:Dlog_to_D} is an isomorphism for all
\(p^n\)-torsion objects, it suffices to treat the case \(pM=0\).
Indeed, we argue by induction on \(n\). Consider the exact sequence
\[
0\longrightarrow p^{n-1}M
\longrightarrow M
\longrightarrow M/p^{n-1}M
\longrightarrow0.
\]
The adapted basis of \(M\) induces adapted bases of
\(p^{n-1}M\) and \(M/p^{n-1}M\). Applying the functors $\mathbb D^{\log}_{\Phi}(\cdot)|_{\mathcal U_K}$ and $\mathbb D(\cdot|_{\mathcal U})$, we obtain a commutative diagram:
\[
\xymatrix{0 \ar[r] & \left.\mathbb D^{\log}_{\Phi}(M/p^{n-1}M)\right|_{\mathcal U_K}\ar[r]\ar[d]^{\sim} & \left.\mathbb D^{\log}_{\Phi}(M)\right|_{\mathcal U_K}
\ar[r]\ar[d] & \left.\mathbb D^{\log}_{\Phi}(p^{n-1}M)\right|_{\mathcal U_K} \ar[r]\ar[d]^{\sim} & 0 \\
0 \ar[r] & \mathbb D((M/p^{n-1}M)|_{\mathcal U}) \ar[r] &
\mathbb D(M|_{\mathcal U}) \ar[r] & \mathbb D((p^{n-1}M)|_{\mathcal U}) \ar[r] & 0.
}
\]
 The top row is exact by \autoref{lem:Dlog_Ext_vanishing} and the bottom row is exact by \cite[Theorem~2.4]{Fal89}.
By the induction hypothesis, the left vertical arrow is an
isomorphism, while the right vertical arrow is an isomorphism. Therefore the middle vertical arrow
is also an isomorphism.

Assume therefore that $pM=0$. By
\autoref{lem:Bpluslog_mod_p},
\[
\mathbb D^{\log}_{\Phi}(M)
\xrightarrow{\sim}
\operatorname{Hom}_{\MFbig{R}}(M,E_{\log}).
\]
Using the adapted basis $m_1,\ldots,m_h$ and the notation
$e_i,a_{ij},b_{ij}$ introduced in the proof of
\autoref{lem:Dlog_Ext_vanishing}, a similar calculation identifies
$\mathbb D^{\log}_{\Phi}(M)$ with the set of solutions modulo $p$ in $\overline{\widehat R}_{\log}$ of the congruences $x_i^p \equiv (-p)^{e_i}\sum_{j=1}^h b_{ij}x_j \pmod{p^{e_i+1}}$ for $1\leq i\leq h$, and each solution modulo $p$ lifts uniquely to a solution of \begin{equation}
\label{equ_solution}
x_i^p =(-p)^{e_i}\sum_{j=1}^h b_{ij}x_j ,\qquad 1\leq i\leq h.
\end{equation}
in $\overline{\widehat R}_{\log}$ by the same Newton lifting argument as in \cite[p.~38]{Fal89}. Thus there is a bijection between $\mathbb D^{\log}_{\Phi}(M)$ and the set of solutions of \eqref{equ_solution} in $\overline{\widehat R}_{\log}$.
Similarly, $\mathbb D(M|_{\mathcal U})$ is identified with the set of solutions of \eqref{equ_solution} in $\overline{\widehat R^\circ}$.

It remains to show that every solution of \eqref{equ_solution} in
$\overline{\widehat R^\circ}$ already lies in
$\overline{\widehat R}_{\log}$. Choose a finite normal
$\widehat R$-subalgebra $ R'''\subset\overline{\widehat R}_{\log}$
occurring in the definition of $\overline{\widehat R}_{\log}$ and
containing all the $b_{ij}$ and set
\[
A'''\coloneqq  R'''[X_1,\ldots,X_h]\Big/ \big(
X_i^p-(-p)^{e_i}\sum_{j=1}^h b_{ij}X_j
\big)_{1\leq i\leq h}.
\]
The same argument as in the proof of \autoref{lem:Dlog_Ext_vanishing} implies  $A'''[1/(pT_1\cdots T_s)]$ is finite \'etale over $\widehat R[1/(pT_1\cdots T_s)]$.
Let $(x_1,\ldots,x_h)$ be a solution of \eqref{equ_solution} in
$\overline{\widehat R^\circ}$. Evaluation at this solution gives
\[
A'''\longrightarrow\overline{\widehat R^\circ}, \qquad
X_i\longmapsto x_i.
\]
After inverting $pT_1\cdots T_s$, this map factors through a connected component $B'''$ of the ring  $A'''[1/(pT_1\cdots T_s)]$.
The integral closure $C'''$ of $\widehat R$ in the fraction field of $B'''$ is therefore one
of the finite extensions occurring in the definition of
$\overline{\widehat R}_{\log}$. The evaluation morphism then factors through $C'''$, hence through $\overline{\widehat R}_{\log}$. Thus each $x_i\in\overline{\widehat R}_{\log}$.  This completes the proof of (1).

For (2), the pair $(Y',Y'\cap Z)$ is logarithmically small. By \cite[Lemma~2.18]{BS22}, the Frobenius lifting $\Phi$ uniquely extends to $\mathcal Y'$. The extended Frobenius lifting also satisfies the
conditions in \autoref{setup:Aff_log_FF}, and we continue to denote it by $\Phi$.

After choosing compatible algebraic closures, the inclusion
$\widehat R\to\widehat R'$ induces an inclusion $\overline{\widehat R}_{\log}\to\overline{\widehat R'}_{\log}$, and hence a homomorphism
$B^+_{\log}(\widehat R)\to B^+_{\log}(\widehat R')$.
The constructions of $\kappa_{\Phi,\log}$ are compatible with this homomorphism. Therefore there is a canonical isomorphism
\[
\mathbb P^{\cris}_{\Phi,\log}(M)
\otimes_{B^+_{\log}(\widehat R)}
B^+_{\log}(\widehat R') \cong
\mathbb P^{\cris}_{\Phi,\log}(M|_{\mathcal Y'}), \]
compatible with the filtrations, Frobenius structures, and the $\pi_1^{\et}(\mathcal Y_K^{\prime\circ})$-actions. Consequently, we obtain a natural $\pi_1^{\et}(\mathcal Y_K^{\prime\circ})$-equivariant morphism
\[\mathbb D^{\log}_{\Phi}(M)|_{\mathcal Y_K^{\prime\circ}}
\longrightarrow \mathbb D^{\log}_{\Phi}(M|_{\mathcal Y'}). \]

It remains to show that this morphism is an isomorphism. It is enough to check this after forgetting the group actions. By (1), the underlying $\mathbb Z_p$-modules of the source and the target are canonically identified with $\mathbb D(M|_{\mathcal U})$ and $\mathbb D(M|_{\mathcal U\cap\mathcal Y'})$, respectively. Under these identifications, the above morphism coincides with the canonical morphism
\[
\mathbb D(M|_{\mathcal U}) \longrightarrow
\mathbb D(M|_{\mathcal U\cap\mathcal Y'}),
\]
coming from the construction of Faltings' $\mathbb D$-functor. Hence the original $\pi_1^{\et}(\mathcal Y_K^{\prime\circ})$-equivariant morphism is an isomorphism. This proves (2).
\end{proof}

\begin{remark}
If $M$ is $p^n$-torsion for some $n\geq1$, then by part~(1) of the proposition and the finiteness of Faltings' $\mathbb D$-functor, $\mathbb D^{\log}_{\Phi}(M)$ is a finite abelian group. The action of $\pi_1^{\et}(\mathcal Y_K^\circ)$ on $\mathbb D^{\log}_{\Phi}(M)$ is continuous, as the actions on $\mathbb P^{\cris}_{\Phi,\log}(M)$ and $D_{\log}(\widehat R)$ are continuous. In the torsion-free case, the continuity of the action follows by passing to the inverse limit over $M/p^nM$.
\end{remark}

We finally show that $\mathbb D^{\log}_{\Phi}(M)$ is independent of the choice of the logarithmic Frobenius lifting.

\begin{proposition}
\label{prop:Dlog_independent_Phi}
Let $\Phi_1$, $\Phi_2$, and $\Phi_3$ be three logarithmic Frobenius liftings satisfying the conditions in \autoref{setup:Aff_log_FF}.
For $i=1,2,3$, let
\[
M_i\in
\mathcal{MF}_{[0,p-2]}^{\nabla,\Phi_i}((\mathcal Y,\mathcal Z)/W)
\text{ or }
\mathcal{MF}_{[0,p-2]}^{\nabla,\mathrm{tf},\Phi_i}((\mathcal Y,\mathcal Z)/W),
\]
where the $M_i$ are either all torsion or all torsion-free. Assume that they are compatible with one another under the equivalences of \autoref{thm_log_version_gluing}. Then, for every $i,j\in\{1,2,3\}$, there exists a canonical $\pi_1^{\et}(\mathcal Y_K^\circ)$-equivariant isomorphism
\[
f_{ij}\colon
\mathbb D_{\Phi_i}^{\log}(M_i)
\xrightarrow{\sim}
\mathbb D_{\Phi_j}^{\log}(M_j)
\]
satisfying the cocycle condition $f_{13}=f_{23}\circ f_{12}$.
\end{proposition}

\begin{proof}
We assume that $M_i=(V,\nabla,\Fil,\varphi_i)$ is of $p^n$-torsion. The torsion-free case then follows by passing to inverse limits.
Via
\autoref{thm_log_version_gluing}, we have $\varphi_j =\varphi_i\circ\alpha_{\Phi_j,\Phi_i}$.
To construct $f_{ij}$, it suffices to build a $\pi_1^{\et}(\mathcal Y_K^\circ)$-equivariant isomorphism
\[
\tau_{ij}: \mathbb P^{\cris}_{\Phi_i,\log}(M_i) \xrightarrow{\sim}\mathbb P^{\cris}_{\Phi_j,\log}(M_j)
\]
preserving the filtration and Frobenius endomorphism, and satisfying the cocycle condition.
For any $v\in V$, adopting the notation  $\nabla(\partial)^{\{I\}}$ from the proof of \autoref{thm_log_version_gluing}, we define $\tau_{ij}$ explicitly by
\begin{equation}\label{equ_Taylor}
    \tau_{ij}(v\otimes1) =
\sum_{I=(i_1,\ldots,i_d)\in\mathbb N^d}
\nabla(\partial)^{\{I\}}(v)\otimes\frac{\delta^I}{I!},
\end{equation}
where $\delta^I \coloneqq \prod_{k=1}^d \delta_k^{i_k}$ with
\[
\delta_k =
\begin{cases}
\kappa_{\Phi_i,\log}(T_k)/\kappa_{\Phi_j,\log}(T_k) - 1, & 1\le k\le s,\\
\kappa_{\Phi_i,\log}(T_k)-\kappa_{\Phi_j,\log}(T_k), & s+1\le k\le d.
\end{cases}
\]
\eqref{equ_Taylor} is well-defined. Indeed, by \autoref{lem:kappa_Phi_log}(3),
$\kappa_{\Phi_i,\log}(T_k)/
\kappa_{\Phi_j,\log}(T_k)-1\in {I_{\theta}}$ for $1\leq k\leq s$, while
$\kappa_{\Phi_i,\log}(T_k)-\kappa_{\Phi_j,\log}(T_k)\in {I_{\theta}}$ for
$s+1\leq k\leq d$. Consequently, the coefficient $\delta^I/I!$ belongs to the $|I|$-th divided power of ${I_{\theta}}$.
Since $B^+_{\log}(\widehat R)$ is complete for the topology defined by $I_{\theta}^{[n]}B^+_{\log}(\widehat R)$, the terms in \eqref{equ_Taylor} tend to zero as $|I|\to\infty$. Thus the series converges and \eqref{equ_Taylor} is well defined. It remains to check the following properties.
\begin{itemize}
    \item $\tau_{ij}$ preserves the filtration. Let  $v\in\Fil^rV$ with $0\leq r\leq p-2$. By Griffiths transversality, each factor appearing in $\nabla(\partial)^{\{I\}}$ lowers the filtration by at most one. Hence
\[
\nabla(\partial)^{\{I\}}(v)\in
\Fil^{\max\{0,r-|I|\}}V.
\]
Set $m=r-\max\{0,r-|I|\}=\min\{|I|,r\}$. Since $r\leq p-2$, it follows that $m\leq p-2$. Hence,
\[
\frac{\delta^I}{I!}\in I_\theta^{[|I|]}B^+_{\log}(\widehat R)
\subseteq I_\theta^{[m]}B^+_{\log}(\widehat R)=
\Fil^m B^+_{\log}(\widehat R).
\]
Therefore, by the tensor product filtration,
\[
\nabla(\partial)^{\{I\}}(v)\otimes\frac{\delta^I}{I!}
\in
\Fil^{\max\{0,r-|I|\}+m}\mathbb P^{\cris}_{\Phi_j,\log}(M)
=
\Fil^r\mathbb P^{\cris}_{\Phi_j,\log}(M).
\]
We conclude that $\tau_{ij}$ preserves the filtration. \item $\tau_{ij}$ satisfies the cocycle condition. The argument is identical to that in the proof of \autoref{thm_log_version_gluing}.

\item $\tau_{ij}$ is $\pi_1^{\et}(\mathcal Y_K^\circ)$-equivariant.
By the definition of the $\pi_1^{\et}(\mathcal Y_K^\circ)$-action in
\eqref{equ:log_Galois_Taylor} and the composition property of the
logarithmic Taylor isomorphisms proved in
\autoref{thm_log_version_gluing}, it suffices to verify that, for every
$\sigma\in\pi_1^{\et}(\mathcal Y_K^\circ)$,
\[
\begin{cases}
\displaystyle
\bigl(1+\Delta_{\Phi_j,\log,k}(\sigma)\bigr)
\bigl(1+\sigma(\delta_k)\bigr)
=
\bigl(1+\delta_k\bigr)
\bigl(1+\Delta_{\Phi_i,\log,k}(\sigma)\bigr),
& 1\leq k\leq s,\\[10pt]
\displaystyle
\Delta_{\Phi_j,\log,k}(\sigma)+\sigma(\delta_k)
=
\delta_k+\Delta_{\Phi_i,\log,k}(\sigma),
& s+1\leq k\leq d.
\end{cases}
\]
These identities follow directly from the definitions.

\item $\tau_{ij}$ is compatible with Frobenius.
Since $\tau_{ij}$ preserves the filtration, it induces an isomorphism $\widetilde{\tau}_{ij}$ on the corresponding tilde modules.
The tensor product filtration gives natural identifications
\[
\widetilde{\mathbb P^{\cris}_{\Phi_i,\log}(M_i)}\cong
\widetilde V\otimes_{\widehat R,\kappa_{\Phi_i,\log}}
\widetilde{B^+_{\log}(\widehat R)}\]
and similarly for $j$.
By \autoref{lem:kappa_Phi_log}(1),
\[
\Phi_B\circ\kappa_{\Phi_i,\log}=\kappa_{\Phi_i,\log}\circ\Phi_i,\qquad \Phi_B\circ\kappa_{\Phi_j,\log}=\kappa_{\Phi_j,\log}\circ\Phi_j.
\]
Hence there are natural identifications
\[
\widetilde{\mathbb P^{\cris}_{\Phi_i,\log}(M_i)}\otimes_{\Phi_B}B^+_{\log}(\widehat R)\cong\bigl(\widetilde V\otimes_{\Phi_i}\widehat R\bigr)\otimes_{\kappa_{\Phi_i,\log}}\left(\widetilde{B^+_{\log}(\widehat R)}\otimes_{\Phi_B}B^+_{\log}(\widehat R)\right)
\]
and similarly for $j$.
Under the above identification, the Frobenius map on $\mathbb P^{\cris}_{\Phi_i,\log}(M_i)$ is $\varphi_i\otimes \varphi_B$ and similarly for $j$.
It therefore suffices to show that the following diagram commutes:
\[
\xymatrix@R=3em@C=5em{
\widetilde{\mathbb P^{\cris}_{\Phi_i,\log}(M_i)}
\otimes_{\Phi_B}B^+_{\log}(\widehat R)\ar[r]^-{\widetilde{\tau}_{ij}\otimes_{\Phi_B}\operatorname{id}}
\ar[d]_{\varphi_{\mathbb P,i}}
&
\widetilde{\mathbb P^{\cris}_{\Phi_j,\log}(M_j)}
\otimes_{\Phi_B}B^+_{\log}(\widehat R)
\ar[d]^{\varphi_{\mathbb P,j}}
\\
\mathbb P^{\cris}_{\Phi_i,\log}(M_i)
\ar[r]^-{\tau_{ij}}
&
\mathbb P^{\cris}_{\Phi_j,\log}(M_j).
}
\]
Under the above identifications, after applying $\operatorname{id}\otimes\varphi_B$, the map $\widetilde{\tau}_{ij}\otimes_{\Phi_B}\operatorname{id}$ becomes the logarithmic Taylor isomorphism associated with
\[
\kappa_{\Phi_i,\log}\circ\Phi_i
\qquad\text{and}\qquad
\kappa_{\Phi_j,\log}\circ\Phi_j.
\]
After base change by $\kappa_{\Phi_j,\log}$, $\alpha_{\Phi_j,\Phi_i}$ is the logarithmic Taylor isomorphism associated with
\[
\kappa_{\Phi_j,\log}\circ\Phi_j
\qquad\text{and}\qquad
\kappa_{\Phi_j,\log}\circ\Phi_i.
\]
Hence, by the cocycle property of logarithmic Taylor isomorphisms, their composition is the logarithmic Taylor isomorphism associated with
\[
\kappa_{\Phi_i,\log}\circ\Phi_i
\qquad\text{and}\qquad
\kappa_{\Phi_j,\log}\circ\Phi_i.
\]
Since $\varphi_i$ is horizontal, it is compatible with this Taylor isomorphism and with the logarithmic Taylor isomorphism $\tau_{ij}$ associated with $\kappa_{\Phi_i,\log}$ and $\kappa_{\Phi_j,\log}$.
Therefore,
\[
\tau_{ij}\circ (\varphi_i\otimes\operatorname{id})=(\varphi_i\otimes\operatorname{id}) \circ (\alpha_{\Phi_j,\Phi_i}\otimes\operatorname{id})\circ \left(
(\operatorname{id}\otimes\varphi_B)\circ (\widetilde{\tau}_{ij}\otimes_{\Phi_B}\operatorname{id})\right).
\]
Using $\varphi_j=\varphi_i\circ\alpha_{\Phi_j,\Phi_i}$, we obtain
\[
\tau_{ij}\circ\varphi_{\mathbb P,i}=\varphi_{\mathbb P,j}\circ\bigl(\widetilde{\tau}_{ij}\otimes_{\Phi_B}\operatorname{id}\bigr).
\]
Hence the diagram commutes. \qedhere
\end{itemize}
\end{proof}

\subsubsection{The global $\mathbb D^{\log}$-functor}

The gluing construction is essentially the same as in the non-logarithmic case. We keep the notation of \autoref{setup:log_FF}.
Let $\{(Y_i,Z_i)\}_{i\in\Lambda}$ be an affine open covering of the logarithmic smooth pair $(Y,Z)$ such that each $(Y_i,Z_i)$ is logarithmically small, where $Z_i=Z\cap Y_i$. Let $(\mathcal Y_i,\mathcal Z_i)$ be the $p$-adic completion of
$(Y_i,Z_i)$. For each $i\in\Lambda$, choose local coordinates $T^{(i)}=\{T_1^{(i)},\ldots,T_d^{(i)}\}$
on $Y_i$ and a Frobenius lifting $\Phi_i$ on $\mathcal Y_i$ satisfying the conditions in \autoref{setup:Aff_log_FF}.

Let $\mathcal U_i=\mathcal Y_i-\mathcal Z_i$ and $\mathcal Y_{i,K}^\circ=\mathcal Y_{i,K}-\mathcal Z_{i,K}$. Let $M$ be a torsion or torsion-free logarithmic Fontaine--Faltings module with Hodge--Tate weights in $[0,p-2]$.
After refining the logarithmically small affine covering, we may assume that $M|_{\mathcal Y_i}$ admits an adapted basis for every $i$.
By restricting to $\mathcal U$, we obtain a usual Fontaine--Faltings module $M|_{\mathcal U}$, and hence a local system $\mathbb D(M|_{\mathcal U})$ on $\mathcal U_K$. For each $i\in\Lambda$, the local logarithmic functor gives a local system $\mathbb D^{\log}_{\Phi_i}(M|_{\mathcal Y_i})$ on
$\mathcal Y_{i,K}^\circ$, whose restriction to $\mathcal U_{i,K}$ is canonically identified with
$\mathbb D(M|_{\mathcal U})|_{\mathcal U_{i,K}}$ by \autoref{lem_affDLog}(1).

Since $Y$ is separated, $Y_i\cap Y_j$ is affine. Hence
$\bigl(Y_i\cap Y_j,\; Z\cap Y_i\cap Y_j\bigr)$
is logarithmically small. By \autoref{lem_affDLog}(2), there are canonical isomorphisms
\[
\mathbb D^{\log}_{\Phi_i}(M|_{\mathcal Y_i})
|_{\mathcal Y_{i,K}^\circ\cap\mathcal Y_{j,K}^\circ} \cong
\mathbb D^{\log}_{\Phi_i}
(M|_{\mathcal Y_i\cap\mathcal Y_j}) \]
and
\[\mathbb D^{\log}_{\Phi_j}(M|_{\mathcal Y_j})
|_{\mathcal Y_{i,K}^\circ\cap\mathcal Y_{j,K}^\circ}\cong
\mathbb D^{\log}_{\Phi_j} (M|_{\mathcal Y_i\cap\mathcal Y_j}),\]
where $\Phi_i$ and $\Phi_j$ on the right-hand sides denote the respective induced Frobenius liftings on the $p$-adic completion of $Y_i\cap Y_j$. By \autoref{prop:Dlog_independent_Phi}, there is a canonical isomorphism
\[
\mathbb D^{\log}_{\Phi_i}
(M|_{\mathcal Y_i\cap\mathcal Y_j})
\xrightarrow{\sim}
\mathbb D^{\log}_{\Phi_j}
(M|_{\mathcal Y_i\cap\mathcal Y_j}).
\]
Composing these isomorphisms, we obtain
\[
\psi_{ij}:
\mathbb D^{\log}_{\Phi_i}(M|_{\mathcal Y_i})
|_{\mathcal Y_{i,K}^\circ\cap\mathcal Y_{j,K}^\circ}
\xrightarrow{\sim}
\mathbb D^{\log}_{\Phi_j}(M|_{\mathcal Y_j})
|_{\mathcal Y_{i,K}^\circ\cap\mathcal Y_{j,K}^\circ}.
\]
By the restriction compatibility in \autoref{lem_affDLog}(2) and the cocycle condition in \autoref{prop:Dlog_independent_Phi}, the isomorphisms $\psi_{ij}$ satisfy the cocycle condition on triple overlaps. Therefore the local systems $\mathbb D^{\log}_{\Phi_i}(M|_{\mathcal Y_i})$ form descent data with respect to the covering $\{\mathcal Y_{i,K}^\circ\}_{i\in\Lambda}$.

Now, again using \cite[Theorem~4.2.9]{Con06}, we globalize the $\mathbb D^{\log}$-functor. Our goal is to construct a $p$-adic \'etale local system on $\mathcal Y_K^\circ$. At each finite level, this is specified by a collection of finite \'etale covers of $\mathcal Y_K^\circ$. Given our covering $\bigsqcup_i\mathcal Y_i\to\mathcal Y$, we obtain an admissible covering $\bigsqcup_i\mathcal Y_{i,K}^\circ\to\mathcal Y_K^\circ$,
which in particular is faithfully flat and admits local fpqc
quasi-sections in the sense of \cite[Definition~4.2.1]{Con06}.
The local $\mathbb D^{\log}$-functors give compatible finite \'etale covers of the $\mathcal Y_{i,K}^\circ$. Since these covers are finite,
the hypotheses of \cite[Theorem~4.2.9]{Con06} are satisfied, and they descend to finite \'etale covers of $\mathcal Y_K^\circ$. Passing compatibly through the finite levels gives a $p$-adic \'etale local system on $\mathcal Y_K^\circ$.

\begin{theorem}
\label{algebraicalize}
We keep the notation of \autoref{setup:log_FF}.
\begin{enumerate}
\item
There is a functor $\mathbb D^{\log}$ from the category of torsion
(resp. torsion-free) logarithmic Fontaine--Faltings modules over
$(\mathcal Y,\mathcal Z)$ with Hodge--Tate weights in $[0,p-2]$ to the
category of continuous representations of
$\pi_1^{\et}(\mathcal Y_K^\circ)$ on finite $p$-primary torsion
(resp. finite free) $\mathbb Z_p$-modules. It is compatible with
Faltings' usual $\mathbb D$-functor in the following sense: for every
such logarithmic Fontaine--Faltings module $M$, there is a canonical
isomorphism of continuous representations of
$\pi_1^{\et}(\mathcal U_K)$,
\[
\mathbb D^{\log}(M)|_{\mathcal U_K}
\cong
\mathbb D(M|_{\mathcal U}).
\]

\item
If $Y/W$ is proper, then $\mathbb D^{\log}$ can be algebraized to a
functor from the category of torsion (resp. torsion-free) logarithmic
Fontaine--Faltings modules over $(\mathcal Y,\mathcal Z)$ to the
category of continuous representations of
$\pi_1^{\et}(Y_K^\circ)$ on finite $p$-primary torsion
(resp. finite free) $\mathbb Z_p$-modules.
\end{enumerate}
\end{theorem}

\begin{proof}
The first assertion follows by gluing the local
$\mathbb D^{\log}$-functors using the descent data constructed above.
If $Y/W$ is proper, then
$(Y_K^\circ)^{\mathrm{an}}=\mathcal Y_K^\circ$ by
\cite[Theorem~A.3.1]{Con99}. Thus
$\pi_1^{\et}(Y_K^\circ)
\cong\pi_1^{\et}(\mathcal Y_K^\circ)$ by
\cite[Theorem~3.1]{Lut93}. This proves the second assertion.
\end{proof}

\begin{remark}
The logarithmic $\mathbb D^{\log}$-functor extends naturally to logarithmic
Fontaine--Faltings modules with Hodge--Tate weights in any interval
$[a,b]$ satisfying $0\leq b-a\leq p-2$.
Indeed, let $M=(V,\nabla,\Fil,\varphi)$ be such a module.
Define
$M_{\mathrm{new}}
=(V,\nabla,\Fil_{\mathrm{new}},\varphi_{\mathrm{new}})$
by setting
\[
\Fil_{\mathrm{new}}^iV=
\begin{cases}
\Fil^{i+a}V, & 0\leq i\leq b-a,\\
0, & b-a<i\leq p-2,
\end{cases}
\]
where $\varphi_{\mathrm{new}}$ is induced from $\varphi$ by the corresponding
reindexing of $\widetilde V$.
Then $M_{\mathrm{new}}$ is a logarithmic Fontaine--Faltings module with
Hodge--Tate weights in $[0,p-2]$. We define
\[
\mathbb D^{\log}(M)
\coloneqq
\mathbb D^{\log}(M_{\mathrm{new}})(a),
\]
where $(a)$ denotes the $a$-th Tate twist of the representation.
\end{remark}

In other words, we have shown the following. Let $Y$ be a proper smooth scheme over $W$, let $Z$ be a relative simple normal crossing divisor on $Y$, and set $Y_K^\circ\coloneqq Y_K-Z_K$. If $0\leq b-a\leq p-2$, then, for
$*\in\{\emptyset,\mathrm{tf}\}$, there is a functor
\[
\mathbb D^{\log}\colon
\mathcal {MF}_{[a,b]}^{\nabla,*}((\mathcal Y,\mathcal Z)/W)
\longrightarrow
\mathrm{Rep}_{\mathbb Z_p}
\bigl(\pi_1^{\et}(\mathcal Y_K^\circ)\bigr)
\longrightarrow
\mathrm{Rep}_{\mathbb Z_p}
\bigl(\pi_1^{\et}(Y_K^\circ)\bigr),
\]
where the second arrow is an equivalence of categories since $Y/W$ is proper. Representations in the essential image of $\mathbb D^{\log}$
are called \emph{logarithmic crystalline representations with Hodge--Tate weights in $[a,b]$}.

We next record two consequences of the preceding local results that will be used below. Using the Zariski local existence of adapted bases, the restriction compatibility and the Ext-vanishing established above yield the following global statement.

\begin{proposition}
\label{prop:Dlog_restriction_exact}
We keep the notation of \autoref{setup:log_FF}.
\begin{enumerate}
    \item \emph{(Restriction compatibility)} Let $Y'\subseteq Y$ be an open subset, let $\mathcal Y'$ be its
    $p$-adic completion with rigid generic fiber $\mathcal Y'_K$, and set $\mathcal Y_K^{\prime\circ}
    \coloneqq
    \mathcal Y'_K\cap\mathcal Y_K^\circ$.
    For every object $M$ of either $\mathcal{MF}_{[0,p-2]}^{\nabla} ((\mathcal Y,\mathcal Z)/W)$ or $\mathcal{MF}_{[0,p-2]}^{\nabla,\mathrm{tf}} ((\mathcal Y,\mathcal Z)/W)$, there is a canonical isomorphism
    \[
    \mathbb D^{\log}(M)|_{\mathcal Y_K^{\prime\circ}}
    \xrightarrow{\sim}
    \mathbb D^{\log}(M|_{\mathcal Y'}).
    \]

    \item \emph{(Exactness)} $\mathbb D^{\log}$ is exact on both $\mathcal{MF}_{[0,p-2]}^{\nabla} ((\mathcal Y,\mathcal Z)/W)$ and $\mathcal{MF}_{[0,p-2]}^{\nabla,\mathrm{tf}} ((\mathcal Y,\mathcal Z)/W)$.
\end{enumerate}
\end{proposition}

\begin{proof}
For (1), choose a logarithmically small affine open covering
$\{Y_i\}_{i\in\Lambda}$ of $Y$ such that
$M|_{\mathcal Y_i}$ admits an adapted basis for every $i$. By the
construction of $\mathbb D^{\log}(M)$, its restriction to
$\mathcal Y_{i,K}^\circ$ is canonically identified with $\mathbb D^{\log}_{\Phi_i}(M|_{\mathcal Y_i})$.

Choose a logarithmically small affine open covering
$\{Y'_{i'}\}_{i'\in\Lambda'}$ of $Y'$. Since $Y$ is separated,
$Y_i\cap Y'_{i'}$ is affine for every $i\in\Lambda$ and
$i'\in\Lambda'$. After refining the covering if necessary, we may assume
that each $Y'_{i'}$ is contained in some $Y_i$. Since
$M|_{\mathcal Y_i}$ admits an adapted basis, so does
$M|_{\mathcal Y'_{i'}}$. By \autoref{lem_affDLog}(2), there is a canonical
isomorphism
\[
\mathbb D^{\log}_{\Phi_i}(M|_{\mathcal Y_i})
|_{\mathcal Y_{i',K}^{\prime\circ}}
\xrightarrow{\sim}
\mathbb D^{\log}_{\Phi_i}(M|_{\mathcal Y'_{i'}}).
\]
These isomorphisms are compatible on the overlaps by
\autoref{lem_affDLog}(2) and the cocycle condition in
\autoref{prop:Dlog_independent_Phi}. Hence they glue to a canonical
isomorphism
\[
\mathbb D^{\log}(M)|_{\mathcal Y_K^{\prime\circ}}
\xrightarrow{\sim}
\mathbb D^{\log}(M|_{\mathcal Y'}).
\]

For (2), we first consider the torsion case. Let
\[
0\longrightarrow M_1\longrightarrow M_2\longrightarrow M_3
\longrightarrow0
\]
be an exact sequence of torsion logarithmic Fontaine--Faltings modules.
Choose a logarithmically small affine open covering
$\{Y_i\}_{i\in\Lambda}$ such that
$M_1|_{\mathcal Y_i}$, $M_2|_{\mathcal Y_i}$, and
$M_3|_{\mathcal Y_i}$ all admit adapted bases. On each $\mathcal Y_i$,
applying $\operatorname{Hom}_{\MFbig{R_i}}
\bigl(-,D_{\log}(\widehat R_i)\bigr)$
to the restricted exact sequence gives an exact sequence
\[
0\longrightarrow
\mathbb D^{\log}_{\Phi_i}(M_3|_{\mathcal Y_i})
\longrightarrow
\mathbb D^{\log}_{\Phi_i}(M_2|_{\mathcal Y_i})
\longrightarrow
\mathbb D^{\log}_{\Phi_i}(M_1|_{\mathcal Y_i})
\longrightarrow
0
\]
by \autoref{lem:Dlog_Ext_vanishing}. By (1), this sequence is canonically identified with the restriction to $\mathcal Y_{i,K}^\circ$ of
\[
0\longrightarrow
\mathbb D^{\log}(M_3)
\longrightarrow
\mathbb D^{\log}(M_2)
\longrightarrow
\mathbb D^{\log}(M_1)
\longrightarrow0
\]
Since restriction of representations does not change the underlying $\mathbb Z_p$-modules, the latter sequence is exact. The torsion-free case follows by reducing modulo $p^n$ for every
$n\geq1$ and then passing to the inverse limit.
\end{proof}

\subsection{Full faithfulness of $\mathbb D^{\log}$}
\label{subsection:FF_Dlog}
Finally, we prove the full faithfulness of $\mathbb D^{\log}$. The main ingredient is a logarithmic analogue of \cite[Lemma~2.5]{Fal89}, which gives a length estimate for a logarithmic Fontaine--Faltings module from a representation mapping minimally to its associated $\mathbb D^{\log}$-representation.

\begin{lemma}
\label{lem:Dlog_minimal}
We keep the notation in \autoref{setup:Aff_log_FF}. Let $M$ be a $p^n$-torsion logarithmic Fontaine--Faltings module over $(\mathcal Y,\mathcal Z)$ with Hodge--Tate weights in $[0,p-2]$, and let $L$ be a $\mathbb Z/p^n\mathbb Z$-representation of $\pi_1^{\et}(\mathcal Y_K^\circ)$. Suppose that $\lambda\colon L\longrightarrow\mathbb D^{\log}(M)$ is a morphism of representations which does not factor through $\mathbb D^{\log}(M/N)$ for any nonzero logarithmic Fontaine--Faltings submodule $N\subset M$. Then
\[
\operatorname{length}_{R_{(p)}}M_{(p)}\leq \operatorname{length}_{\mathbb Z_p}L.
\]
Moreover, if equality holds, then $\lambda$ is an isomorphism.
\end{lemma}
\begin{proof}
We first reduce to the case $n=1$. Assume that the assertion is known for representations killed by $p$ and, inductively, for those killed by $p^{n-1}$. Let
\[
L[p]=\{\ell\in L\mid p\ell=0\}.
\]
The restriction of $\lambda$ induces a morphism $L[p]\longrightarrow \mathbb D^{\log}(M/pM)$.
Let $M_1\subset M$ be maximal, with $pM\subset M_1$, such that this morphism factors through $\mathbb D^{\log}(M/M_1)$. By the $p$-torsion case,
\[
\operatorname{length}_{R_{(p)}}(M/M_1)_{(p)}
\leq \dim_{\mathbb F_p}L[p].
\]
The composite $L\to\mathbb D^{\log}(M)\to\mathbb D^{\log}(M_1)$ vanishes on $L[p]$, and therefore induces
\[
\overline{\lambda}\colon L/L[p]\longrightarrow\mathbb D^{\log}(M_1).
\]
The minimality assumption for $\lambda$ implies the corresponding minimality assumption for this morphism.
Since $L/L[p]$ is killed by $p^{n-1}$, this morphism factors through
\[
\mathbb D^{\log}(M_1)[p^{n-1}]
=
\mathbb D^{\log}(M_1/p^{n-1}M_1).
\]
Its minimality therefore implies that $p^{n-1}M_1=0$.
Hence, by induction,
\[
\operatorname{length}_{R_{(p)}}(M_1)_{(p)}\leq \operatorname{length}_{\mathbb Z_p}(L/L[p]).
\]
Adding the two inequalities gives
\[
\operatorname{length}_{R_{(p)}}M_{(p)} \leq \operatorname{length}_{\mathbb Z_p}L.
\]
If equality holds, then equality holds in both inequalities, and the induction hypothesis together with the $p$-torsion case shows that $\lambda$ is an isomorphism. Thus it remains to prove the assertion when $pM=pL=0$.

By \autoref{lem:Bpluslog_mod_p},
\[
\mathbb D^{\log}(M) \cong
\operatorname{Hom}_{\MFbig{R}}(M,E_{\log}).
\]
Hence $\lambda$ induces an $R/pR$-linear map
\[
\alpha\colon
M\longrightarrow \operatorname{Hom}_{\mathbb F_p}(L,E_{\log}),
\qquad
\alpha(m)(\ell)=\lambda(\ell)(m).
\]
Here $\operatorname{Hom}_{\mathbb F_p}(L,E_{\log})$ is regraded as an
object of $\MFbig{R}$ and is isomorphic to $E_{\log}^{\oplus\dim_{\mathbb F_p}L}$.
By \cite[Proposition~2.2]{Fal89}, $\alpha$ is strict and its kernel $N$ is an object of $\MF{R}$. Using the same argument as in
\cite[pp.~39--40]{Fal89}, $N$ is stable under the logarithmic connection,
and hence is a logarithmic Fontaine--Faltings submodule of $M$.
Since every element in the image of $\lambda$ vanishes on $N$, the morphism $\lambda$ factors through $\mathbb D^{\log}(M/N)$.
By the minimality assumption, $N=0$, and therefore $\alpha$ is injective.

After inverting $T_1\cdots T_s$, we obtain an injection
\[
M|_{\mathcal U}
\longrightarrow
\operatorname{Hom}_{\mathbb F_p}
\bigl(L,E_{\log}[1/(T_1\cdots T_s)]\bigr).
\]
We claim that the natural map
\[
E_{\log}[1/(T_1\cdots T_s)]
\longrightarrow
\overline{\widehat R^\circ}/p\overline{\widehat R^\circ}
\]
induced by
$\overline{\widehat R}_{\log}\hookrightarrow
\overline{\widehat R^\circ}$
is injective.
Since $T_1\cdots T_s$ is invertible in the target, it suffices to consider $x\in\overline{\widehat R}_{\log}$ whose image in $\overline{\widehat R^\circ}/p\overline{\widehat R^\circ}$ is zero. Thus ${x=py}$ for some ${y\in\overline{\widehat R^\circ}}$. Choose a finite normal $\widehat R$-algebra $C\subset\overline{\widehat R}_{\log}$ containing $x$.
Since ${C[1/(T_1\cdots T_s)]}$ is excellent and normal, its $p$-adic completion $C[1/(T_1\cdots T_s)]^{\wedge}$ is normal.  As $y$ is integral over $\widehat R^\circ$, the natural map $\widehat R^\circ \to C[1/(T_1\cdots T_s)]^{\wedge}$ shows that $y$ is integral over $C[1/(T_1\cdots T_s)]^{\wedge}$. Thus the normality of $C[1/(T_1\cdots T_s)]^{\wedge}$ implies $y\in C[1/(T_1\cdots T_s)]^{\wedge}$. Therefore,
\[x\in p C[1/(T_1\cdots T_s)]^{\wedge}\cap C\subseteq p C[1/(T_1\cdots T_s)]\subseteq p \overline{\widehat R}_{\log}[1/(T_1\cdots T_s)].\]
This implies that $E_{\log}[1/(T_1\cdots T_s)] \longrightarrow \overline{\widehat R^\circ}/p\overline{\widehat R^\circ}$ is injective.

Via this injection and the comparison isomorphism of
\autoref{algebraicalize}, the restriction of $\alpha$ gives the ordinary evaluation map
\[
\alpha_{\mathcal U}\colon
M|_{\mathcal U}
\longrightarrow
\operatorname{Hom}_{\mathbb F_p}
\bigl(L,\overline{\widehat R^\circ}/p\overline{\widehat R^\circ}\bigr),
\]
which is injective. This implies that the restriction
$\lambda|_{\mathcal U}\colon
L\longrightarrow \mathbb D(M|_{\mathcal U})$
does not factor through $\mathbb D(M|_{\mathcal U}/N^\circ)$ for any nonzero Fontaine--Faltings submodule
$N^\circ\subseteq M|_{\mathcal U}$.
Therefore \cite[Lemma~2.5]{Fal89} gives
\[
\operatorname{length}_{R^\circ_{(p)}}
(M|_{\mathcal U})_{(p)}\leq
\dim_{\mathbb F_p}L.
\]
Since $T_1,\ldots,T_s$ are already invertible in $R_{(p)}$, we have $R^\circ_{(p)}=R_{(p)}$. Hence
\[
\operatorname{length}_{R_{(p)}}M_{(p)}\leq\dim_{\mathbb F_p}L.
\]

If equality holds, then \cite[Lemma~2.5]{Fal89} implies that
$\lambda|_{\mathcal U}$ is an isomorphism. Since restriction does not change the underlying $\mathbb F_p$-linear map, $\lambda$ itself is an isomorphism.
\end{proof}

Using this logarithmic analogue of Faltings' lemma, we now prove the full faithfulness of $\mathbb D^{\log}$.

\begin{proposition}\label{prop_FullyFaithful}
The functor $\mathbb D^{\log}$ is fully faithful on both the torsion and the torsion-free categories of logarithmic Fontaine--Faltings modules with Hodge--Tate weights in $[0,p-2]$.
\end{proposition}
\begin{proof}
The proof proceeds in three steps.

\noindent
\emph{Step 1: Reduction to the torsion and local case.}
By reducing modulo $p^n$ and passing to the inverse limit, full faithfulness for torsion-free logarithmic Fontaine--Faltings modules follows from the torsion case. Hence it suffices to treat torsion objects.

Assume that $\mathbb D^{\log}$ is fully faithful over every logarithmically small affine open subset of $\mathcal Y$.
Let $f\colon M_1\to M_2$ be a morphism such that $\mathbb D^{\log}(f)=0$. For every logarithmically small affine open subset $\mathcal V\subset\mathcal Y$, we have $\mathbb D^{\log}
\bigl(f|_{\mathcal V}\bigr)=0$ by \autoref{prop:Dlog_restriction_exact}(1).
The local full faithfulness implies that $f|_{\mathcal V}=0$. Hence $f=0$.

Now let $\alpha\colon
\mathbb D^{\log}(M_2)
\longrightarrow
\mathbb D^{\log}(M_1)$
be a morphism of $\pi_1^{\et}(\mathcal Y_K^\circ)$-representations. For every logarithmically small affine open subset $\mathcal V$,
local full faithfulness yields a unique morphism
\[
f_{\mathcal V}\colon
M_1|_{(\mathcal V,\mathcal V\cap\mathcal Z)}
\longrightarrow
M_2|_{(\mathcal V,\mathcal V\cap\mathcal Z)}
\]
such that $\mathbb D^{\log}(f_{\mathcal V}) =\alpha|_{\mathcal V^\circ_K}$ by \autoref{prop:Dlog_restriction_exact}(1).
If $\mathcal V$ and $\mathcal V'$ are two such open subsets, then on $\mathcal V_K \cap \mathcal V'_K$,
\[
\mathbb D^{\log}(f_{\mathcal V}|_{\mathcal V \cap \mathcal V'}) = \alpha|_{\mathcal V^\circ_K \cap \mathcal V'^\circ_K} = \mathbb D^{\log}(f_{\mathcal V'}|_{\mathcal V \cap \mathcal V'}),
\]
and local full faithfulness implies $f_{\mathcal V}|_{\mathcal V \cap \mathcal V'} = f_{\mathcal V'}|_{\mathcal V \cap \mathcal V'}$.
Thus the morphisms $f_{\mathcal V}$ glue to a global morphism
$f\colon M_1\to M_2$ satisfying
$\mathbb D^{\log}(f)=\alpha$.

Therefore, to prove that $\mathbb D^{\log}$ is fully faithful, it suffices to establish full faithfulness for the local $\mathbb D^{\log}$ functor. In what follows, we adopt the notation from \autoref{setup:Aff_log_FF}.
\medskip

\noindent
\emph{Step 2: Faithfulness.}
Let $f:M_1\to M_2$ be a morphism of logarithmic Fontaine--Faltings modules satisfying $\mathbb D^{\log}(f)=0$. By the comparison isomorphism in
\autoref{algebraicalize}, $\mathbb D(f|_{\mathcal U})=0$. Since the usual $\mathbb D$-functor is faithful, $f|_{\mathcal U}=0$.

For every $n\geq1$, after forgetting the connection,
$f\bmod p^n$ is a morphism in $\MF{R}$. On the one hand, since $\MF{R}$ is an abelian category,
\[
\operatorname{Im}(f\bmod p^n)\in\MF{R}.
\]
Its localization at $T_1\cdots T_s$ is zero. On the other hand, the underlying $\widehat{R}$-module of $\operatorname{Im}(f\bmod p^n)$ is Zariski locally a direct sum of modules of the form $R/p^eR$. Since $T_1\cdots T_s$ is a non-zero-divisor on each $R/p^eR$, it follows that
\[
\operatorname{Im}(f\bmod p^n)=0.
\]
Thus $f\bmod p^n=0$ for every $n$, and hence $f=0$. Therefore $\mathbb D^{\log}$ is faithful.

\medskip

\noindent
\emph{Step 3: Fullness.}
It is enough to prove the assertion for $p^n$-torsion objects. Indeed, the torsion-free case follows by applying the torsion case to $M/p^nM$ and passing to the inverse limit. The resulting morphisms are compatible by the faithfulness proved in Step~2.
Let $M_1$ and $M_2$ be $p^n$-torsion logarithmic Fontaine--Faltings modules, and let
\[
\alpha:
\mathbb D^{\log}(M_2)
\longrightarrow
\mathbb D^{\log}(M_1)
\]
be a morphism of representations. For every torsion logarithmic Fontaine--Faltings module
$M$, we have
\begin{equation}\label{equ:Dlog_length}
\operatorname{length}_{\mathbb Z_p}\mathbb D^{\log}(M) =
\operatorname{length}_{\mathbb Z_p}\mathbb D(M|_{\mathcal U})=
\operatorname{length}_{R_{(p)}}M_{(p)}.
\end{equation}
where the first equality follows from the comparison isomorphism in \autoref{algebraicalize}, and the second follows from \cite[Theorem~2.4]{Fal89}, together with $R^\circ_{(p)}=R_{(p)}$. Consider
\[
(\alpha,\operatorname{id}):
\mathbb D^{\log}(M_2)
\longrightarrow
\mathbb D^{\log}(M_1\oplus M_2).
\]
Choose a maximal logarithmic Fontaine--Faltings submodule
$N\subset M_1\oplus M_2$ such that this morphism factors through $\mathbb D^{\log}(Q)$, where $Q=(M_1\oplus M_2)/N$.
Write
\[
\lambda:
\mathbb D^{\log}(M_2)\longrightarrow\mathbb D^{\log}(Q)
\]
for the induced morphism and $q:M_1\oplus M_2\to Q$ for the quotient map. By maximality, $\lambda$ satisfies the hypothesis of
\autoref{lem:Dlog_minimal}, hence
\[
\operatorname{length}_{R_{(p)}}Q_{(p)}
\leq
\operatorname{length}_{\mathbb Z_p}\mathbb D^{\log}(M_2)
=
\operatorname{length}_{R_{(p)}}(M_2)_{(p)}.
\]

Let $q_2:M_2\to Q$ be the restriction of $q$. Since $\mathbb D^{\log}(q_2)\circ\lambda = \operatorname{id}$, the map $\mathbb D^{\log}(q_2)$ is surjective. By
\eqref{equ:Dlog_length}, the opposite inequality follows, and therefore
\[
\operatorname{length}_{R_{(p)}}Q_{(p)}=\operatorname{length}_{R_{(p)}}(M_2)_{(p)}.
\]
The equality statement in \autoref{lem:Dlog_minimal} shows that
$\lambda$ is an isomorphism. Hence $\mathbb D^{\log}(q_2)$ is an isomorphism. Since $\mathbb D^{\log}$ is exact by \autoref{prop:Dlog_restriction_exact}(2) and faithful by Step~2, $q_2$ is an isomorphism. If $q_1:M_1\to Q$ denotes the other component of $q$, define
\[
f=q_2^{-1}\circ q_1:M_1\longrightarrow M_2.
\]
We can check that $ \mathbb D^{\log}(f)=\alpha$. Thus $\mathbb D^{\log}$ is full.
\end{proof}

\ \\
\textbf{Conflict of Interest:} The authors declare that they have no conflict of interest.

\end{document}